\documentclass[reqno,10pt]{amsart}

\usepackage{amsmath,amssymb,amsthm,mathrsfs,graphicx,enumitem,verbatim}

\usepackage{xcolor}
\definecolor{winered}{rgb}{0.7,0,0}
\definecolor{lessblue}{rgb}{0,0,0.7}

\usepackage[pdftex=true,
	bookmarks=true,
	bookmarksopen=true,
	bookmarksnumbered=true,
	pdftoolbar=true,
	pdfmenubar=true,
	pdffitwindow=false,
	pdfstartview={FitH},
	colorlinks=true,
	linkcolor=winered,
	citecolor=lessblue,
	breaklinks=true]{hyperref}

\allowdisplaybreaks

\numberwithin{equation}{section}
\newtheorem{thm}{Theorem}[section]
\newtheorem{prop}[thm]{Proposition}
\newtheorem{lemma}[thm]{Lemma}
\newtheorem{cor}[thm]{Corollary}
\newtheorem*{thm*}{Theorem}
\newtheorem*{prop*}{Proposition}
\newtheorem*{cor*}{Corollary}
\newtheorem*{conj*}{Conjecture}

\theoremstyle{definition}
\newtheorem{definition}[thm]{Definition}

\theoremstyle{remark}
\newtheorem{rmk}[thm]{Remark}

\newcommand{\mc}{\mathcal}

\newcommand{\cC}{\mc C}

\newcommand{\cH}{\mc H}

\newcommand{\cL}{\mc L}
\newcommand{\cM}{\mc M}

\newcommand{\cO}{\mc O}
\newcommand{\cP}{\mc P}
\newcommand{\cQ}{\mc Q}
\newcommand{\cR}{\mc R}

\newcommand{\cU}{\mc U}
\newcommand{\cV}{\mc V}
\newcommand{\cW}{\mc W}
\newcommand{\cX}{\mc X}
\newcommand{\cY}{\mc Y}

\newcommand{\cCH}{\cC\cH}

\newcommand{\ms}{\mathscr}

\newcommand{\sS}{\ms S}

\newcommand{\C}{\mathbb{C}}
\newcommand{\N}{\mathbb{N}}
\newcommand{\R}{\mathbb{R}}
\newcommand{\Z}{\mathbb{Z}}
\newcommand{\Sph}{\mathbb{S}}

\newcommand{\sfp}{\mathsf{p}}

\newcommand{\ran}{\operatorname{ran}}

\renewcommand{\Re}{\operatorname{Re}}
\renewcommand{\Im}{\operatorname{Im}}

\newcommand{\supp}{\operatorname{supp}}
\newcommand{\sgn}{\operatorname{sgn}}

\newcommand{\tr}{\operatorname{tr}}

\newcommand{\rank}{\operatorname{rank}}

\newcommand{\la}{\langle}
\newcommand{\ra}{\rangle}
\newcommand{\pa}{\partial}
\newcommand{\tn}{\textnormal}

\newcommand{\eps}{\epsilon}

\newcommand{\wt}{\widetilde}
\newcommand{\wh}{\widehat}
\newcommand{\ol}{\overline}

\newcommand{\hra}{\hookrightarrow}

\newcommand{\bop}{{\mathrm{b}}}

\newcommand{\cp}{{\mathrm{c}}}

\newcommand{\psdo}{ps.d.o.}
\newcommand{\Diff}{\mathrm{Diff}}

\newcommand{\Diffb}{\Diff_\bop}

\newcommand{\Psib}{\Psi_\bop}

\newcommand{\WF}{\mathrm{WF}}

\newcommand{\WFb}{\WF_{\bop}}

\newcommand{\Tb}{{}^{\bop}T}
\newcommand{\rcTb}{{}^{\bop}\overline{T}}

\newcommand{\Sb}{{}^{\bop}S}
\newcommand{\Nb}{{}^{\bop}N}

\newcommand{\half}{\frac{1}{2}}

\newcommand{\ham}{H}
\newcommand{\rham}{{\mathsf{H}}}

\newcommand{\bhm}{M_\bullet}

\newcommand{\loc}{{\mathrm{loc}}}
\newcommand{\CI}{\cC^\infty}

\newcommand{\CIc}{\cC^\infty_\cp}
\newcommand{\CmI}{\cC^{-\infty}}

\newcommand{\Hloc}{H_{\loc}}

\newcommand{\Hb}{H_{\bop}}
\newcommand{\Hbloc}{H_{\bop,\loc}}

\newcommand{\Rnhalfc}{{\overline{\R^n_+}}}

\newcommand{\openbigpmatrix}[1]{\addtolength{\arraycolsep}{-#1}\begin{pmatrix}}
\newcommand{\closebigpmatrix}[1]{\end{pmatrix}\addtolength{\arraycolsep}{#1}}

\newcommand{\itref}[1]{(\ref{#1})}

\newlength{\enummargin}
\begin{document}

\title[Decay at the Cauchy horizon of Kerr]{Boundedness and decay of scalar waves at the Cauchy horizon of the Kerr spacetime}
\author{Peter Hintz}

\address{Department of Mathematics, University of California, Berkeley, CA 94720-3840, USA}
\email{phintz@berkeley.edu}


\date{December 25, 2015. Last revision: February 28, 2017.}

\begin{abstract}
  Adapting and extending the techniques developed in recent work with Vasy for the study of the Cauchy horizon of cosmological spacetimes, we obtain boundedness, regularity and decay of linear scalar waves on subextremal Reissner--Nordstr\"om and (slowly rotating) Kerr spacetimes, without any symmetry assumptions; in particular, we provide simple microlocal and scattering theoretic proofs of analogous results by Franzen. We show polynomial decay of linear waves relative to a Sobolev space of order slightly above $1/2$. This complements the generic $\Hloc^1$ blow-up result of Luk and Oh.
\end{abstract}

\maketitle

\section{Introduction}
\label{SecIntro}

We analyze regularity and decay of linear scalar waves near the Cauchy horizon of asymptotically flat black hole spacetimes by adapting and extending microlocal and scattering theoretic arguments used in recent work with Vasy \cite{HintzVasyCauchyHorizon} on cosmological black hole spacetimes; in particular, this provides new and independent proofs of boundedness and $\cC^0$ extendability results obtained by Franzen \cite{FranzenRNBoundedness}, see also the work in progress \cite{FranzenKerrBoundedness} and the discussion below. The spacetimes we consider are Reissner--Nordstr\"om spacetimes, i.e.\ non-rotating black holes, with non-zero charge, and Kerr spacetimes, i.e.\ rotating black holes, with non-zero angular momentum; see Figure~\ref{FigIntroPenrose} for their Penrose diagrams. Near the region we are interested in, these spacetimes are Lorentzian 4-manifolds with the topology $\R_{t_0}\times(0,\infty)_r\times\Sph^2_\omega$, equipped with a Lorentzian metric $g$ of signature $(1,3)$. They have two horizons, namely the \emph{Cauchy horizon} $\cCH^+$ at $r=r_1$ and the \emph{event horizon} $\cH^+$ at $r=r_2>r_1$. In order to quantify decay rates, we use a time function $t_0$, which is equivalent to the Boyer--Lindquist coordinate $t$ away from the event and Cauchy horizons, i.e.\ $t_0$ differs from $t$ by a smooth function of the radial coordinate $r$; and $t_0$ is equivalent to the Eddington--Finkelstein coordinate $u$ near the Cauchy horizon, and to the Eddington--Finkelstein coordinate $v$ near the event horizon. (Since we are interested in the part of the spacetime close to $\cCH^+$, the choice of $t_0$ for large $r$ is irrelevant.) We consider the Cauchy problem for the linear wave equation with Cauchy data posed on a surface $H_I$ as indicated in Figure~\ref{FigIntroPenrose}.

\begin{figure}[!ht]
  \centering
  \includegraphics{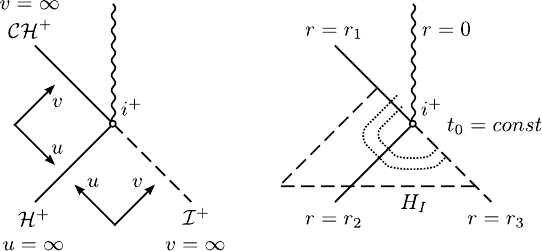}
  \caption{\textit{Left:} Penrose diagram of the Reissner--Nordstr\"om spacetime, and of an $\omega=const$ slice of a Kerr spacetime with angular momentum $a\neq 0$; shown are the Cauchy horizon $\cCH^+$ and the event horizon $\cH^+$, as well as future timelike infinity $i^+$. The coordinates $u,v$ are Eddington--Finkelstein coordinates. \textit{Right:} The same Penrose diagram. The region enclosed by the dashed lines is the domain of dependence of the Cauchy surface $H_I$. The dotted lines are two level sets of the function $t_0$; the smaller one of these corresponds to a larger value of $t_0$.}
\label{FigIntroPenrose}
\end{figure}

In its simplest form, the main result of the paper states:

\begin{thm}
\label{ThmIntroMain}
  Let $g$ be a subextremal Reissner--Nordstr\"om or slowly rotating Kerr metric with non-zero angular momentum. Suppose $u$ solves the Cauchy problem for the wave equation $\Box_g u=0$ with smooth compactly supported initial data. Then for all $\eps>0$, there exists a constant $C>0$ such that
  \begin{equation}
  \label{EqIntroMainBound}
    |u| \leq C t_0^{-2+\eps}
  \end{equation}
  uniformly in $r>r_1$.
  
  More precisely, near $\cCH^+$, the solution $u(t_0,r,\omega)$ is smooth in $(t_0,\omega)$, with values in the Sobolev space $H^{1/2+(1/2+\eps')\log}((r_1,r_1+\delta))$, and norm bounded by $C t_0^{-2+\eps}$, for $\eps,\delta>0$ and sufficiently small $\eps'>0$.
\end{thm}

Here, the Sobolev space $H^{s+\ell\log}(\R)$ consists of all $u$ for which $\la\xi\ra^s\la\log\la\xi\ra\ra^\ell\hat u(\xi)\in L^2$, where $\la\xi\ra=(1+|\xi|^2)^{1/2}$, and \eqref{EqIntroMainBound} follows from the second part of the Theorem via the Sobolev embedding $H^{1/2+(1/2+\eps'\log)}(\R)\hra L^\infty(\R)$ for $\eps'>0$.

We shall in fact prove that a bound $t_0^{-\alpha}$, $\alpha>3/2$, for $u$ and its derivatives in the exterior region implies a uniform bound for $u$, and its derivatives tangential to the Cauchy horizon, by $t_0^{-\alpha+3/2+\eps}$ near the Cauchy horizon; and if $D_{t_0}u$ and its derivatives are bounded by $t_0^{-\alpha-1}$, we obtain the stronger bound $t_0^{-\alpha+1+\eps}$ near $\cCH^+$; see Theorem~\ref{ThmPfMain}. The bound \eqref{EqIntroMainBound} is then a consequence of Price's law for scalar waves, which gives $\alpha=3$; Price's law has been proved rigorously for all subextremal Reissner--Nordstr\"om spacetimes and slowly rotating Kerr spacetimes by Tataru \cite{TataruDecayAsympFlat} and Metcalfe, Tataru, and Tohaneanu \cite{MetcalfeTataruTohaneanuPriceNonstationary}. These are in fact black box results, hence the uniform boundedness of energy and the integrated local energy decay estimates proved by Dafermos, Rodnianski, and Shlapentokh-Rothman \cite{DafermosRodnianskiShlapentokhRothmanDecay} imply Price's law for the full subextremal Kerr family. We also mention the work of Dafermos and Rodnianski \cite{DafermosRodnianskiPrice} for a proof of Price's law in a nonlinear, but spherically symmetric setting. In order to make our arguments as simple as possible, we assume the angular momentum to be small in a technical step in \S\ref{SubsecPfAsy}; however, a slight improvement of the argument will likely eliminate this restriction, see Remark~\ref{RmkPfAsySmallA}.

Our arguments also apply to the initial value problem with smooth, compactly supported Cauchy data posed on a \emph{two-ended} hypersurface: In this case, we obtain uniform regularity up to the bifurcate Cauchy horizon and the bifurcation sphere. (See e.g.\ \cite[Figure~1]{LukOhReissnerNordstrom}.)

As a consequence of our main theorem, we recover Franzen's result on the boundedness and $\cC^0$ extendibility of linear waves on Reissner--Nordstr\"om spacetimes, and are moreover able to extend it to the case of slowly rotating Kerr spacetimes:

\begin{thm}[Franzen \cite{FranzenRNBoundedness}, and ongoing work \cite{FranzenKerrBoundedness}]
\label{ThmIntroBoundedness}
  Under the assumptions of Theorem~\ref{ThmIntroMain}, $u$ remains uniformly bounded in the black hole interior, and moreover $u$ extends by continuously to $\cCH^+$.
\end{thm}

We point out however that Franzen's results apply under assumptions on the decay rate in the exterior region which are less restrictive than what we need for our arguments; and the techniques used (based on vector field multipliers) are rather different than the ones used in the present paper. We also remark that one can directly infer some (weak) decay along the Cauchy horizon from her work. For the nonlinear Einstein vacuum equations, Dafermos and Luk recently announced the $\cC^0$ stability of the Cauchy horizon for perturbations of Kerr spacetimes, assuming suitable decay rates (compatible with Price's law) of the perturbation to a Kerr spacetime in the exterior region \cite{DafermosICM2014}.

On the other hand, Luk and Oh prove the following blow-up result:\footnote{Since the appearance of the first version of the present paper, Luk and Sbierski \cite{LukSbierskiKerr} have proved a corresponding blow-up result on Kerr spacetimes, subject to an assumption on a lower bound along the event horizon. See also the related work by Dafermos and Shlapentokh-Rothman \cite{DafermosShlapentokhRothmanBlueShift}.}

\begin{thm}[Luk-Oh \cite{LukOhReissnerNordstrom}]
\label{ThmIntroBlowUp}
  The solution to the Cauchy problem $\Box_g u=0$ fails to be in $\Hloc^1$ near any point on the Cauchy horizon for generic $\CI$ initial data with compact support.
\end{thm}

Our main result therefore provides a rather precise boundedness and regularity estimates, complementing this blow-up statement. We expect that the loss in the decay rate, mainly caused by the conversion between $L^2$ and $L^\infty$ decay rates, can be eliminated by using more precise, $L^\infty$-type Besov spaces.

In sharp contrast to Theorem~\ref{ThmIntroBlowUp}, Gajic \cite{GajicExtremalRN}, following work of Aretakis \cite{AretakisExtremalRN1,AretakisExtremalRN2} and using the upcoming \cite{AngelopoulosAretakisGajicExtremalRN}, has recently shown in the case of \emph{extremal} Reissner--Nordstr\"om spacetimes that $u$ extends in $\Hloc^1$ past $\cCH^+$; we will however not discuss extremal black holes in the present paper.

\subsection{Sketch of the proof}
\label{SubsecIntroSketch}

The presence of an asymptotically flat spatial infinity, rather than an asymptotically hyperbolic one as in cosmological spacetimes (i.e.\ in the presence of a cosmological constant $\Lambda>0$), causes difficulties in the analysis of the stationary operator $\wh\Box_g(\sigma):=e^{it_0\sigma}\Box_g e^{-it_0\sigma}$ near the real axis, and specifically near $\sigma=0$. (In fact, it is a precise analysis of the behavior of $\wh\Box_g(\sigma)^{-1}$ as $\Im\sigma\to 0+$ which is key to Tataru's proof of Price's law \cite{TataruDecayAsympFlat}.) However, assuming suitable decay in the exterior region, we can circumvent these difficulties. In the notation of Theorem~\ref{ThmIntroMain}, we thus proceed as follows:

\medskip

\textit{Step 1.} (See \S\ref{SubsecPfLoc}.) We cut the solution $u$ off at a large radius; the cut-off wave $u_\chi$ solves $\Box_g u_\chi=f_\chi$, where $f_\chi\in t_0^{-3}L^\infty$ has the same decay properties as the original $u$. Translating this into $L^2$ spacetime decay gives $u_\chi\in t_0^{-5/2+\eps}L^2$.

\medskip

\textit{Step 2.1.} (See \S\S\ref{SubsecPfMod} and \ref{SubsecPfSol}.) We are now free to modify the spacetime near spatial infinity by adding a cosmological horizon, which eliminates the aforementioned low frequency problems.

\medskip

We want to use methods of microlocal analysis near the Cauchy horizon, using the saddle point structure of the null-geodesic flow in a suitable uniform version of phase space, called b-cotangent bundle, on the compactification of the spacetime at $t_0=\infty$, in the spirit of \cite[\S2]{HintzVasySemilinear}; see Figures~\ref{FigIntroExt} and \ref{FigPfSolFlow} below. It is then technically very convenient to extend the spacetime and the wave equation under consideration beyond $\cCH^+$ so that the Cauchy horizon itself lies in the interior of the domain on which we solve the thus extended equation.

\medskip

\textit{Step 2.2} We modify the spacetime beyond $\cCH^+$ by adding an `artificial exterior region,' including an artificial horizon $\ol\cH^a$; denote the metric of the extended spacetime by $\wt g$. See Figure~\ref{FigIntroExt} for a Penrose diagram. (For the purpose of hiding the possibly complicated structure of the extension, we place a complex absorbing operator $\cQ\in\Psi^2$ behind $\cCH^+$ in the spirit of \cite{NonnenmacherZworskiQuantumDecay,WunschZworskiNormHypResolvent}. We will drop $\cQ$ in this brief sketch.) We can then solve the Cauchy problem for the equation $\Box_{\wt g}\wt u=f_\chi$, where we allow $\wt u$ a priori to have exponential growth. By uniqueness for the wave equation, we have $\wt u=u_\chi$ in $r>r_1$; the goal now is to combine this information with precise regularity estimates for $\wt u$ at the Cauchy horizon to obtain uniform bounds and decay for $\wt u$, and thus for $u_\chi$.

\begin{figure}[!ht]
  \centering
  \includegraphics{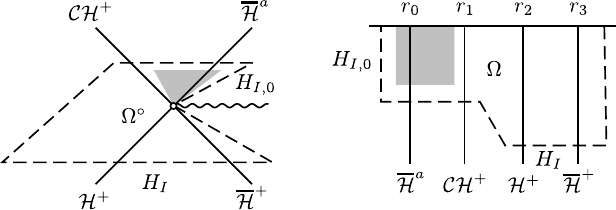}
  \caption{\textit{Left:} Penrose diagram of the extended spacetime. We add a cosmological horizon $\ol\cH^+$ at $r=r_0$ as well as an artificial horizon $\ol\cH^a$ at $r=r_0$, creating an artificial exterior region $r_0<r<r_1$. We place complex absorption in the shaded region. In order to have uniqueness for the wave equation in the domain $\Omega^\circ$, we place Cauchy data at $H_I\cup H_{I,0}$. \textit{Right:} The compactification of the spacetime, and of the domain $\Omega^\circ$, at $t_0=\infty$.}
\label{FigIntroExt}
\end{figure}

\medskip

This strategy was recently implemented by Vasy and the author \cite{HintzVasyCauchyHorizon} in the cosmological setting, i.e.\ for charged and/or rotating black holes in a de Sitter universe, in order to obtain asymptotic expansions, decay and regularity results for waves near the Cauchy horizon of cosmological spacetimes.

\medskip

\textit{Step 3.} (See \S\ref{SubsecPfAsy}.) One shows next that $\wt u$ in fact lies in $t_0^{-5/2+\eps}H^N$ near the Cauchy horizon for some (negative) $N$; furthermore, one obtains $t_0^{-5/2+\eps}H^N$ for \emph{all} $N$ away from the Cauchy horizon. (We thus show using scattering theory how polynomial decay propagates from the exterior region into the black hole exterior in the `no-shift region' in the terminology of \cite{FranzenRNBoundedness}; this was previously observed by Luk \cite{LukSchwarzschild} for the Schwarzschild black hole.) Using previously known radial point (microlocal blue-shift) estimates \cite[Proposition~2.1]{HintzVasySemilinear}, one can take $N=1/2-\delta$, $\delta>0$, at $\cCH^+$, which however does not quite give an $L^\infty$ bound. (The only issue is regularity across the horizon; spherical and time derivatives of $\wt u$ remain in $H^{1/2-\delta}$ with the same decay rate.)

\medskip

\textit{Step 4.} (See \S\ref{SubsecPfRad}.) The regularity at $\cCH^+$ can be crucially improved using radial point estimates at the Cauchy horizon. The technical heart of the present paper is thus a logarithmic improvement of radial point estimates at the borderline regularity; see Theorem~\ref{ThmRadialLogOut}. (This estimate is conveniently formulated on the compactified spacetime, and makes use of Melrose's b-calculus \cite{MelroseAPS}, both conceptually and technically.) This allows us to obtain the regularity $H^{1/2+(1/2+\delta)\log}$ for $\wt u$ at $\cCH^+$ upon giving up an additional $t_0^{1+\delta}$ of decay. The Sobolev embedding of the space $H^{1/2+(1/2+\delta)\log}(\R)$ into $L^\infty(\R)$ thus yields the uniform bound for $u$ in $t_0^{-3/2+\eps}L^\infty$, $\eps>0$, since we have smoothness in $t_0$ and in the spherical variables. Assuming better decay for $D_{t_0}u$ in the exterior region, we can integrate the analogous result for $D_{t_0}u$ from infinity and obtain the sharper result $u\in t_0^{-2+\eps}L^\infty$.

\medskip

In order to obtain the regularity of $\wt u$ with respect to $D_{t_0}$ and spherical derivatives, we need to use the precise structure of the spacetime: Regularity with respect to $D_{t_0}$ follows from the stationarity of the metric, while regularity in the spherical variables uses spherical symmetry in the Reissner--Nordstr\"om case and hidden symmetries in the Kerr case. Namely, in the latter case, we make use of the Carter operator, introduced in \cite{CarterKillingTensor} and intimately related to the Carter tensor \cite{WalkerPenroseType22} and the Carter constant \cite{CarterGlobalKerr}, and used extensively in the proof of wave decay in the exterior region of slowly rotating Kerr spacetimes by Andersson and Blue \cite{AnderssonBlueHiddenKerr}. Since the existence of the Carter constant is directly tied to the completely integrable nature of the geodesic flow on Kerr, this part of the argument is rather inflexible; see however Remark~\ref{RmkRadialConormal} for further discussion.

\subsection{Motivation and previous work}
\label{SubsecIntroMot}

The study of linear scalar waves serves as a toy model for understanding the problem of determinism in Einstein's theory of general relativity: Charged Reissner--Nordstr\"om and rotating Kerr solutions extend analytically beyond the Cauchy horizon, and in fact there are many inequivalent smooth extensions; an observer on such a fixed background spacetime could cross $\cCH^+$ in finite time and enter a region of spacetime where the metric tensor is not uniquely determined, given complete knowledge of the initial data. Motivated by heuristic arguments of Simpson and Penrose \cite{SimpsonPenroseBlueShift}, Penrose formulated the Strong Cosmic Censorship conjecture, which asserts that maximally globally hyperbolic developments for the Einstein--Maxwell or Einstein vacuum equations (depending on whether one considers charged or uncharged solutions) with \emph{generic} initial data (and a complete initial surface, and/or under further conditions) are inextendible as suitably regular Lorentzian manifolds. In particular, the smooth extendability of the Reissner--Nordstr\"om and Kerr solutions past their Cauchy horizons is conjecturally an unstable phenomenon. We refer to works by Christodoulou \cite{ChristodoulouInstabililtyOfNakedSing}, Dafermos \cite{DafermosEinsteinMaxwellScalarStability,DafermosInterior,DafermosBlackHoleNoSingularities}, and Costa, Gir\~ao, Nat\'ario, and Silva \cite{CostaGiraoNatarioSilvaCauchy1,CostaGiraoNatarioSilvaCauchy2,CostaGiraoNatarioSilvaCauchy3} in the spherically symmetric setting for positive and negative results for various notions of regularity, and to work in progress by Dafermos and Luk \cite{DafermosICM2014} on the $C^0$ stability of the Kerr Cauchy horizon. For an overview of the history of this line of study, we refer the reader to the excellent introductions of \cite{DafermosBlackHoleNoSingularities,LukOhReissnerNordstrom}.

In the black hole interior, the linear scalar wave equation was studied by several authors \cite{FranzenRNBoundedness,GajicExtremalRN,LukOhReissnerNordstrom,SbierskiThesis} using vector field methods, and by Vasy and the author \cite{HintzVasyCauchyHorizon} in the cosmological setting using scattering theory and microlocal analysis, using the framework developed by Vasy \cite{VasyMicroKerrdS} and extended in \cite{BaskinVasyWunschRadMink,HintzVasySemilinear}; the insight that the methods of \cite{HintzVasyCauchyHorizon} could be improved to work in the asymptotically flat case is what led to the present paper.

The analysis of linear and non-linear waves in the exterior region of asymptotically flat spacetimes has a very rich history, see \cite{KayWaldSchwarzschild,BachelotSchwarzschildScattering,DafermosEinsteinMaxwellScalarStability,DafermosRodnianskiPrice,DafermosRodnianskiLectureNotes,MarzuolaMetcalfeTataruTohaneanuStrichartz,DonningerSchlagSofferPrice,TohaneanuKerrStrichartz,TataruDecayAsympFlat,SterbenzTataruMaxwellSchwarzschild,AnderssonBlueMaxwellKerr,DafermosRodnianskiShlapentokhRothmanDecay} and references therein. In the asymptotically hyperbolic case of cosmological spacetimes, methods of scattering theory have proven very useful in this context, see \cite{SaBarretoZworskiResonances,BonyHaefnerDecay,DyatlovQNM,DyatlovQNMExtended,WunschZworskiNormHypResolvent,VasyMicroKerrdS,MelroseSaBarretoVasySdS,HintzVasyKdsFormResonances} and references therein; the present work is inspired by the philosophy underlying these latter works.

\subsection*{Acknowledgments}
\label{SubsecIntroAck}

I wish to thank Anne Franzen, Jonathan Luk, Sung-Jin Oh, Andr\'as Vasy, and Maciej Zworski for very helpful discussions. I would also like to thank an anonymous referee for comments which improved the exposition of the  paper. I am grateful for the hospitality of the Erwin Schr\"odinger Institute in Vienna, where part of this work was carried out. I gratefully acknowledge support by the Miller Institute at the University of California, Berkeley.

\section{Setup and proof of the main theorem}
\label{SecPf}

We recall the form of the metric of the Reissner--Nordstr\"om--de Sitter family of black holes,
\begin{equation}
\label{EqPfRNdSMetric}
  g_{\bhm,Q,\Lambda} = \mu\,dt^2 - \mu^{-1}\,dr^2 - r^2\,d\omega^2,\quad \mu=1-\frac{2\bhm}{r}+\frac{Q^2}{r^2}-\frac{\Lambda r^2}{3},
\end{equation}
on $\R_t\times I_r\times\Sph^2$, $I\subset\R$ open, with $d\omega^2$ the round metric on $\Sph^2$. Here, $\bhm$ is the mass of the black hole, $Q$ its charge, and $\Lambda\geq 0$ is the cosmological constant. We are interested in the Reissner--Nordstr\"om family $g_{\bhm,Q}:=g_{\bhm,Q,0}$, which is called \emph{subextremal} for the parameter range $|Q|<\bhm$: In this case, $\mu(r)$ has two unique roots $0<r_1<r_2$, called \emph{Cauchy horizon} ($r=r_1$) and \emph{event horizon} ($r=r_2$). The metric, when written in the coordinates \eqref{EqPfRNdSMetric}, has singularities at the horizons, which can be removed by a change of coordinates. Concretely, defining
\begin{equation}
\label{EqPfRNdSChange}
  t_0 = t - F(r)
\end{equation}
with $F$ smooth on $\R\setminus\{r_1,r_2\}$, $F'(r)=s_j\mu^{-1}$ near $r=r_j$, where $s_j=-\sgn\mu'(r_j)$, and $F\equiv 0$ for $r>r_2+1$, the metric $g$ in the coordinates $(t_0,r,\omega)$ is stationary ($\pa_{t_0}$ is Killing) and a smooth non-degenerate Lorentzian metric of signature $(1,3)$ on
\[
  M^\circ := \R_{t_0}\times(0,\infty)_r\times\Sph^2_\omega.
\]
See Figure~\ref{FigPfRN}. In terms of Eddington--Finkelstein coordinates $u,v$, we have $t_0=u$ Near the Cauchy horizon and $t_0=v$ near the event horizon.

\begin{figure}[!ht]
  \centering
  \includegraphics{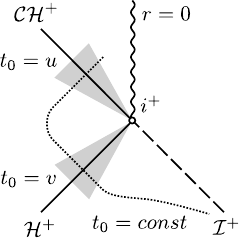}
  \caption{Penrose diagram of the Reissner--Nordstr\"om spacetime. Shown is a level set of the coordinate $t_0$, and the regions (shaded) near the event, resp.\ Cauchy horizon, where $t_0$ is equal to one of the Eddington--Finkelstein coordinates, $t_0=v$, resp.\ $t_0=u$.}
\label{FigPfRN}
\end{figure}

Our results also apply to the Kerr family of black holes; the metric of the more general Kerr--de Sitter family has the form
\begin{equation}
\label{EqPfKdSMetric}
\begin{split}
  g_{\bhm,a,\Lambda} &= -\rho^2\Bigl(\frac{dr^2}{\wt\mu}+\frac{d\theta^2}{\kappa}\Bigr) - \frac{\kappa\sin^2\theta}{(1+\gamma)^2\rho^2}(a\,dt-(r^2+a^2)\,d\phi)^2 \\
    &\qquad + \frac{\wt\mu}{(1+\gamma)^2\rho^2}(dt-a\sin^2\theta\,d\phi)^2,
\end{split}
\end{equation}
where
\begin{gather*}
  \wt\mu = (r^2+a^2)\Bigl(1-\frac{\Lambda r^2}{3}\Bigr) - 2\bhm r, \\
  \rho^2 = r^2+a^2\cos^2\theta, \quad \gamma = \frac{\Lambda a^2}{3}, \quad \kappa = 1+\gamma\cos^2\theta;
\end{gather*}
here, $a$ is the angular momentum. The Kerr spacetimes are the subfamily $g_{\bhm,a}:=g_{\bhm,a,0}$, and the \emph{subextremal} spacetimes are those for which $|a|<\bhm$. For $a\neq 0$, the metric $g_{\bhm,a}$ again has coordinate singularities at the two roots $0<r_1<r_2$ of $\wt\mu$, called Cauchy and event horizon, respectively; these singularities can be resolved by defining
\begin{equation}
\label{EqPfKdSChange}
  t_0 = t - F(r),\quad \phi_0 = \phi - P(r),
\end{equation}
where
\[
  F'(r) = s_j\frac{r^2+a^2}{\wt\mu},\quad P'(r)=s_j\frac{a}{\wt\mu}
\]
near $r=r_j$; here $s_j=-\sgn\mu'(r_j)$. Thus, in the coordinates $(t_0,r,\phi_0,\theta)$, the metric $g_{\bhm,a}$ is a smooth non-degenerate Lorentzian metric on
\[
  M^\circ := \R_{t_0}\times(0,\infty)_r\times\Sph^2_{\phi_0,\theta}.
\]
The singularity of the metric at $\theta=0,\pi$ is merely a coordinate singularity, related to the singular nature of the standard spherical coordinates on $\Sph^2$ at the poles; see \cite[\S6.1]{VasyMicroKerrdS}. The Penrose diagram of a constant $(\phi_0,\theta)$ slice of the spacetime is the same as the one depicted in Figure~\ref{FigPfRN}.

For both classes of spacetimes, we can modify the `time function' $t_0$ by putting $t_*=t_0-\wt F(r)$, with $\wt F$ smooth, such that $dt_*$ is spacelike for $r\leq r_1+2\delta$ and $r\geq r_2-2\delta$, where $0<\delta<(r_2-r_1)/4$ is any fixed number.

In order to succinctly state our main result, we define the collections of vector fields
\begin{equation}
\label{EqPfVf}
  \cM_0 = \{ \pa_{t_*}, \Omega_1, \Omega_2, \Omega_3 \}, \quad \cM = \{\pa_r\}\cup\cM_0,
\end{equation}
where the $\Omega_j$, which are vector fields on $\Sph^2$, span the tangent space of $\Sph^2$ at every point. Moreover, we recall from \S\ref{SecIntro} the space
\[
  H^{s+\ell\log}(\R^n) = \{ u \colon \la\xi\ra^s\la\log\la\xi\ra\ra^\ell\wh u(\xi)\in L^2(\R^n_\xi) \}.
\]

\begin{thm}
\label{ThmPfMain}
  Let $g$ be the metric of a subextremal Reissner--Nordstr\"om spacetime or a Kerr spacetime with $a\neq 0$ very small. Let $u$ be the solution of the Cauchy problem for $\Box_g u=0$ in $r>r_1$, with initial data which are smooth and compactly supported in $\{t_*=0,\ r>r_2-\delta\}$. Fix a radius $r_+>r_2$. Suppose $\alpha>3/2$ is such that for all $\eta>0$, $N\in\Z_{\geq 0}$, all vector fields $V_1,\ldots,V_N\in\cM$, there exists a constant $C_\eta$ such that the estimate
  \begin{equation}
  \label{EqPfMainAssReg}
    |V_1\cdots V_N u(r,t_*)| \leq C_\eta t_*^{-\alpha}, \quad r_2+\eta\leq r\leq r_+ + 1,
  \end{equation}
  holds. Then for every $\eps>0$, $N\in\Z_{\geq 0}$, and all vector fields $V_1,\ldots,V_N\in\cM_0$, we have 
  \begin{equation}
  \label{EqPfMainResReg}
    V_1\cdots V_N u \in t_*^{-\alpha+3/2+\eps}H^{1/2+(1/2+\eps)\log}
  \end{equation}
  near $r=r_1$. In particular, $u$ extends continuously to $\cCH^+$, and for every $\eps>0$, there exists a constant $C$ such that the uniform estimate
  \begin{equation}
  \label{EqPfMainResBdd}
    |u(r,t_*)|\leq Ct_*^{-\alpha+3/2+\eps},\quad r_1<r<r_1+1
  \end{equation}
  holds.
  
  If we make the assumption
  \begin{equation}
  \label{EqPfMainAssReg2}
    |V_1\cdots V_N D_{t_*}u(r,t_*)| \leq C_\eta t_*^{-\alpha-1}, \quad r_2+\eta\leq r\leq r_+ + 1
  \end{equation}
  in addition to \eqref{EqPfMainAssReg}, then we have the stronger conclusion
  \begin{equation}
  \label{EqPfMainResBdd2}
    |u(r,t_*)|\leq Ct_*^{-\alpha+1+\eps}, \quad r_1<r<r_1+1.
  \end{equation}
\end{thm}

\begin{rmk}
\label{RmkPfFiniteReg}
  A close inspection of the proof shows that if we assume \eqref{EqPfMainAssReg} for $N=0,\ldots,N_0$, then  one obtains \eqref{EqPfMainResReg} for $N=0,\ldots,N_0-2$, and likewise for the additional assumption \eqref{EqPfMainAssReg2} and the stronger conclusion \eqref{EqPfMainResReg} with $-\alpha+3/2+\eps$ replaced by $-\alpha+1+\eps$. In particular, in order to obtain the $L^\infty$ decay statement \eqref{EqPfMainResBdd}, we need to control $N_0=4$ derivatives of $u$ in the exterior domain, while for \eqref{EqPfMainResBdd2}, we need $N_0=5$ derivatives (one of which is $D_{t_*}$).
\end{rmk}

The proof of Theorem~\ref{ThmPfMain}, subdivided into several steps, proceeds along the lines outlined in \S\ref{SubsecIntroSketch}, and is given in \S\S\ref{SubsecPfLoc}--\ref{SubsecPfRad}.

\emph{We will use the notation of Theorem~\ref{ThmPfMain} throughout the remainder of this section.}

\subsection{Localization away from spatial infinity}
\label{SubsecPfLoc}

Let us study the forcing problem
\[
  \Box_g u = f \in \CIc(M^\circ);
\]
any initial value problem can easily be converted into this problem, with $f$ supported close to the Cauchy surface. Then, let $\chi\in\CI(\R)$ be a cutoff function,
\[
  \chi(r)\equiv 1,\quad r\leq r_+,\qquad \chi(r)\equiv 0,\quad r\geq r_+ + 1.
\]
Define
\begin{equation}
\label{EqPfLocUchiFchi}
  u_\chi := \chi u,\quad f_\chi := \chi f + [\Box_g,\chi]u,
\end{equation}
so $\Box_g u_\chi=f_\chi$. Since $f$ has compact support, the assumption \eqref{EqPfMainAssReg} on $u$ implies that $u_\chi$ and $f_\chi$ satisfy the same estimates. Thus, $u_\chi$ has $L^2$ decay
\[
  V_1\ldots V_N u_\chi \in t_*^{-\alpha+1/2+\eps}L^2(\R_{t_*}\times(0,\infty)_r\times\Sph^2),\quad \eps>0,
\]
likewise for $f_\chi$; this in turn can be rewritten as
\begin{equation}
\label{EqPfLocL2Decay}
  u_\chi,f_\chi \in t_*^{-\alpha+1/2+\eps}H^\infty.
\end{equation}

\subsection{Modification and extension of the spacetime}
\label{SubsecPfMod}

In order to understand regularity and decay of $u_\chi$ uniformly up to the Cauchy horizon, we aim to construct a suitable extension of the spacetime in consideration beyond the Cauchy horizon. On the extension, we will then solve a wave-type equation in \S\ref{SubsecPfSol} whose solution $\wt u$ will agree with $u_\chi$ in $r>r_1$; properties of $\wt u$ near $r=r_1$ then give the corresponding uniform properties of $u_\chi$ in $r>r_1$.

First, we replace the asymptotically flat end at $r=\infty$ by an asymptotically hyperbolic one. Concretely, let $r_{m,+}>r_+$ be such that the trapped set (i.e.\ the projection of the trapped set in the cotangent bundle of $M^\circ$ to the base) lies in $r<r_{m,+}$. Choose a cutoff function $\chi_m\in\CI(\R)$ with $\chi_m(r)\equiv 1$, $r\leq r_{m,+}$, and $\chi_m(r)\equiv 0$, $r\geq r_{m,+}+1$. For $\Lambda>0$ and in the Reissner--Nordstr\"om case $g=g_{\bhm,Q}$, we define
\[
  \wt g := \chi_m g_{\bhm,Q} + (1-\chi_m)g_{\bhm,Q,\Lambda},
\]
while in the Kerr case $g=g_{\bhm,a}$, we take
\[
  \wt g := \chi_m g_{\bhm,a} + (1-\chi_m)g_{\bhm,a,\Lambda}.
\]
For sufficiently small $\Lambda$, the tensor $\wt g$ is then again a Lorentzian metric, and now there is an additional \emph{cosmological horizon} at $r=r_3$, which is the largest positive root of the function $\mu_{\bhm,Q,\Lambda}$ in the Reissner--Nordstr\"om--de Sitter case \eqref{EqPfRNdSMetric}, and of $\wt\mu_{\bhm,a,\Lambda}$ in the Kerr--de Sitter case \eqref{EqPfKdSMetric}, where we made the spacetime parameters explicit as subscripts. We remark that in the Kerr case, $\wt g$ is a small perturbation of the Kerr--de Sitter metric in $\{r_2-\delta<r<r_3+\delta\}$, and hence for small $a$ a small perturbation of Schwarzschild--de~Sitter space in this region, which we will use in \S\ref{SubsecPfAsy}. Defining $t_0$ (and $\phi_0$ in the Kerr case) near $r=r_3$ as in \eqref{EqPfRNdSChange} (resp.\ \eqref{EqPfKdSChange}), we can again extend the metric past the coordinate singularity at $r=r_3$. The dynamics of the null-geodesic flow in the region $r_1<r<r_3+1$ are completely analogous to the dynamics of the null-geodesic flow on Reissner--Nordstr\"om--de Sitter or Kerr--de Sitter space in the past of the Cauchy horizon; see also Figures~\ref{FigPfSolFlow} and \ref{FigPfSolFlow2} below.

To make the region near $r=r_1$ amenable to our analysis, we modify the function $\mu$ (resp.\ $\wt\mu$): Fix $\delta>0$ and $r_0,r_{\cQ,\pm}$ such that $0<r_0-\delta<r_{\cQ,-}<r_0<r_{\cQ,+}<r_1$, and let $\mu_*$ (resp.\ $\wt\mu_*$) be a smooth function, equal to $\mu$ (resp.\ $\wt\mu$) in $r\geq r_{\cQ,+}$, but such that it has exactly one additional simple root at $r=r_0$; thus $s_0=-\sgn\mu'(r_0)=-1$ (resp.\ $s_0=-\sgn\wt\mu'(r_0)=-1$). In $r\leq r_1$, we then let $\wt g$ be equal to the metric \eqref{EqPfRNdSMetric} with $\mu$ replaced by $\mu_*$ (resp.\ \eqref{EqPfKdSMetric}, with $\wt\mu$ replaced by $\wt\mu_*$). This creates an \emph{artificial horizon} $\ol\cH^a$ at $r=r_0$ on the modified spacetime
\[
  \wt M^\circ = \R_{t_0}\times(r_0-2\delta,r_3+2\delta)_r\times\Sph^2.
\]
We can again modify $t_0$ by letting $t_*=t_0-\wt F(r)$, $\wt F$ smooth on $\R$, such that $dt_*$ is timelike for $r\in[r_0-2\delta,r_1+2\delta]$ and $r\in[r_2-2\delta,r_3+2\delta]$; see \cite[\S2]{HintzVasyCauchyHorizon} for details. See Figure~\ref{FigPfModPenrose} for the Penrose diagram of the modified spacetime.

\begin{figure}[!ht]
  \centering
  \includegraphics{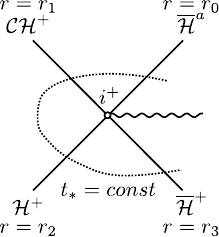}
  \caption{Penrose diagram of the modified and extended spacetime $\wt M^\circ$. We add a cosmological horizon $\ol\cH^+$ at $r=r_3$ and an artificial horizon $\ol\cH^a$ at $r=r_0$, beyond $\cCH^+$. The dotted line is a level set of the coordinate $t_*$, which is spacelike for $r_j-2\delta\leq r\leq r_{j+1}+2\delta$, $j=0,2$.}
\label{FigPfModPenrose}
\end{figure}

The spacetime $(\wt M^\circ,\wt g)$ is time-orientable: We can define a time-orientation by letting $dt_*$ be future timelike in $r\geq r_2-2\delta$; then $-dr$ is future timelike in $r_1<r<r_2$, and $-dt_*$ is future timelike in $r\leq r_1+2\delta$.

The dynamics in the region $r_0-\delta<r<r_1$ (with $\delta>0$ small) will be irrelevant below since we will use a complex absorbing operator there, while the dynamics near $r=r_1$ are dictated by the fact that $r=r_1$ defines a horizon with non-zero surface gravity.

\begin{rmk}
  In the Reissner--Nordstr\"om case, we could in addition require $\mu_*$ to have a single non-degenerate critical point at $r_{P,*}\in(r_0,r_1)$; then, the qualitative behavior of $\mu_*$ in a neighborhood of $[r_0,r_1]$ is the same as that near $[r_2,r_3]$, so the region $r_0-\delta<r<r_1+\delta$ has the same null-geodesic dynamics as a neighborhood of the exterior a Schwarzschild--de Sitter black hole, and one may reasonably call $r_0<r<r_1$ an `artificial exterior region.' Due to the more delicate algebra for the trapping in Kerr, such an intuition is more difficult to arrange and justify.
\end{rmk}

\begin{rmk}
\label{RmkPfModSpatialInfty}
  The modification near spatial infinity is to a large extent a matter of convenience, since it allows us to use standard results \cite{VasyMicroKerrdS} on the meromorphic properties of the inverse of the stationary operator $\wh{\Box_{\wt g}}(\sigma)$, which is closely related to the spectral family of an asymptotically hyperbolic metric; see also \cite{MazzeoMelroseHyp,SaBarretoZworskiResonances,GuillarmouMeromorphic,MelroseSaBarretoVasyResolvent} for the analysis of such metrics. Without this modification, one would have to deal with an asymptotically flat end, for which the analysis of $\wh\Box_g(\sigma)^{-1}$ is rather delicate, see e.g.\ \cite{TataruDecayAsympFlat} and the references therein.
\end{rmk}

\subsection{Construction and solution of the extended problem}
\label{SubsecPfSol}

Within $\wt M^\circ$, we now define the surface
\[
  H_I := \{ r_2-\delta\leq r\leq r_3+\delta,\ t_*=0 \},
\]
which  is the Cauchy surface in Theorem~\ref{ThmPfMain}. We further define the artificial surfaces
\begin{gather*}
  H_{F,3} := \{ r=r_3+\delta,\ t_*\geq 0 \}, \\
  H_{F,2} := \{ r_1+\delta\leq r\leq r_2-\delta,\ t_*=C(r_2-\delta-r) \},
\end{gather*}
which are both spacelike (in the case of $H_{F,2}$, we need to choose $C>0$ sufficiently large); in the artificially extended region, we lastly define, using $t_{*,0}=C(r_2-r_1-2\delta)$ (which is the value of $t_*$ at the points of $H_{F,2}$ with smallest $r$):
\begin{equation}
\label{EqPfSolWhiteHoleSurf}
\begin{gathered}
  H_F := \{ r_0-\delta\leq r\leq r_1+\delta,\ t_*=t_{*,0} \}, \\
  H_{I,0} := \{ r=r_0-\delta,\ t_*\geq t_{*,0} \}.
\end{gathered}
\end{equation}
These hypersurfaces bound a submanifold with corners $\Omega^\circ\subset\wt M^\circ$. See Figure~\ref{FigPfSolOmega}.

\begin{figure}[!ht]
  \centering
  \includegraphics{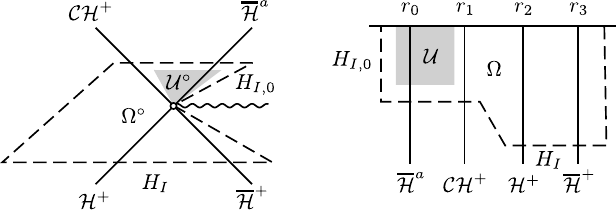}
  \caption{\textit{Left:} The domain $\Omega^\circ$ on which we study an extension of the wave equation. The shaded region $\cU^\circ$ is where we place the complex absorbing operator $\cQ$. \textit{Right:} The compactified picture: We compactify the spacetime $\wt M^\circ$ at $t_*=\infty$, obtaining $\wt M$, and similarly compactify $\Omega^\circ$ to $\Omega$, and $\cU^\circ$ to $\cU$. }
\label{FigPfSolOmega}
\end{figure}

In order to hide the (possibly complicated) null-geodesic dynamics of the metric $\wt g$ in $r<r_1$, we use a complex absorbing operator $\cQ$ as in \cite{HintzVasyCauchyHorizon}, which has Schwartz kernel supported in $\ol{\cU^\circ}\times\ol{\cU^\circ}$, where $\cU^\circ=\{r_{\cQ,-}<r<r_{\cQ,+},\ t_*>t_{*,0}+1\}$. Concretely, in the partial compactification of $\wt M^\circ$ at $t_*=\infty$ defined by
\begin{equation}
\label{EqPfSolComp}
  \wt M := [0,\infty)_\tau \times (r_0-2\delta,r_3+2\delta)_r \times \Sph^2,\quad \tau=e^{-t_*},
\end{equation}
we take $\cQ\in\Psib^2(\wt M)$ to be dilation-invariant near $\tau=0$; more precisely, we assume $\cQ=N(\cQ)$ in $\tau<e^{-t_{*,0}-1}$, where $N(\cQ)$ denotes the normal operator of $\cQ$.

\begin{rmk}
\label{RmkPfSolNoB}
  We can choose $\cQ$ without reference to the b-calculus as follows: Using the coordinates $(t_*,r,\omega,t_*',r',\omega')$ on $\cU\times\cU$, we have $\cQ\in\Psi^2(\wt M^\circ)$, the Schwartz kernel $K_\cQ$ of $\cQ$ is independent of $t_*'-t_*$ for large $t_*$, that is,
\[
  K_\cQ(t_*,r,\omega,t_*',r',\omega') = K'_{N(\cQ)}(t_*'-t_*,r,\omega,r',\omega'),\quad t_*\gg t_{*,0},
\]
for a suitable distribution $K'_{N(\cQ)}$; and $t_*'-t_*$ is bounded on $\supp K_\cQ$, i.e.\ $K_{N(\cQ)}$ vanishes for large $t_*'-t_*$. The normal operator $N(\cQ)$ then has Schwartz kernel
\[
  K_{N(\cQ)}(t_*,r,\omega,t_*',r',\omega') = K'_{N(\cQ)}(t_*'-t_*,r,\omega,r',\omega').
\]
\end{rmk}

We assume that $\cQ$ is elliptic in $\cU=\{r_{\cQ,-}<r<r_{\cQ,+},\ \tau<e^{-t_{*,0}-1}\}$, and further that the sign of the real part of its principal symbol $\sigma(\cQ)$ is non-negative in the future light cone and non-positive in the past light cone.

We furthermore introduce a smooth and stationary potential $V\in\CI(\Omega^\circ)$, which has compact support in the spatial variables $(r,\omega)$, with support contained in the set $\{\chi_m\equiv 0\}$, and so that $V\geq 0$ is positive on a non-empty open set. We then consider the operator
\[
  \cP := \Box_{\wt g} - V - i\cQ
\]
on the domain $\Omega$, which is the closure of $\Omega^\circ$ in $\wt M$; $\cP$ is (b-)pseudo-differential near $\cU$, but differential everywhere else, in particular near the hypersurfaces defined above. The natural function spaces are weighted b-Sobolev spaces,
\[
  \Hb^{s,r}(\Omega) \equiv e^{-r t_*}H^s(\Omega^\circ);
\]
thus, we measure regularity with respect to the vector fields in the set $\cM$ defined in \eqref{EqPfVf}. Let us denote by $X=\Omega\cap\pa\wt M$ the boundary of $\Omega$ at infinity, so
\[
  X = \{ r_1-\delta\leq r\leq r_3+\delta \} \times \Sph^2,
\]
identified with $\{\tau=0\}\times X\subset\wt M$.

As mentioned in \S\ref{SubsecIntroSketch}, the analysis of the extended operator proceeds exactly as in \cite{HintzVasyCauchyHorizon}: Our modified and extended spacetime $\wt M$ has the same geometric and dynamical structure (spacelike boundaries, radial points at the horizons, normally hyperbolic trapping) --- see Figures~\ref{FigPfSolFlow} and \ref{FigPfSolFlow2} ---, and the extended problem has \emph{almost} the same analytic properties (complex absorption), as the cosmological spacetimes considered there, and hence we simply refer to the reference for details. The only, very minor, difference is that we do not assume $\cQ$ to be symmetric here, since we will need to perturb $\cQ$ slightly by a non-symmetric operator to arrange a technical condition, see Lemma~\ref{LemmaPfSolAbsence} below. The analysis is unaffected by this, since the non-vanishing of the real part of $\cQ$ on $\cU^\circ$ guarantees ellipticity there, regardless of the imaginary part.

\begin{figure}[!ht]
  \centering
  \includegraphics{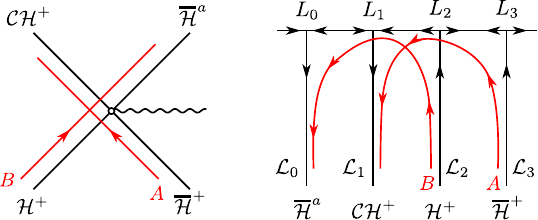}
  \caption{Future-directed null-geodesic flow near the horizons: The null-geodesic flow lifted to the b-cotangent bundle $\Tb^*\wt M$ extends smoothly to the boundary $\tau=0$, and has saddle points $L_j$ at (the b-conormal bundles of) the intersection of $r=r_j$ with the boundary $\tau=0$, with stable (resp.\ unstable) manifolds $\cL_2$, $\cL_3$ (resp.\ $\cL_0$, $\cL_1$). This corresponds to the red-shift effect near $\cH^+$ and $\ol\cH^+$, and the blue-shift effect near $\ol\cH^a$ and $\cCH^+$, which are microlocal \emph{radial point} estimates, see \S\ref{SubsecRadial}.}
\label{FigPfSolFlow}
\end{figure}

\begin{figure}[!ht]
  \centering
  \includegraphics{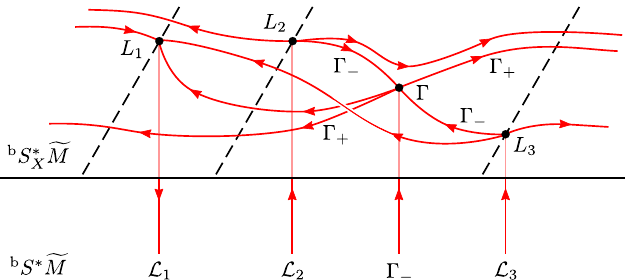}
  \caption{Future-directed null-geodesic flow in phase space $\Sb^*\wt M=\Tb^*\wt M/\R_+$, focusing on the structure of the trapping $\Gamma$, with its stable ($\Gamma_-$) and unstable ($\Gamma_+$) manifolds crossing the horizons $r=r_j$, $j=1,2,3$, transversally, or tending towards the conormal bundle $L_j$ of a horizon.}
\label{FigPfSolFlow2}
\end{figure}

The results which are essential for us are then:

\begin{enumerate}
\item \label{ItPfSolMeromorphic} The operator family $\wh\cP(\sigma)$, obtained by conjugating the normal operator $N(\cP)$ by the Mellin transform in $\tau$ (equivalently, the Fourier transform in $-t_*$), has a meromorphic inverse. (This is why we modified the spacetime to be asymptotically hyperbolic by adding a cosmological horizon.) Moreover, in any fixed strip $-\gamma\leq\Im\sigma\leq C'$ with $\gamma>0$ sufficiently small, $\wh\cP(\sigma)^{-1}$ only has finitely many poles. We also have high energy estimates, which we state in a very non-optimal form which is however sufficient for our purposes: For $|\Im\sigma|<\gamma$, $\gamma>0$ sufficiently small and fixed, and for $\sigma$ a fixed distance away from any pole of $\wh\cP(\sigma)$, we have the following estimate for $s\in\Z_{\geq 0}$ and for $\delta'>0$ arbitrary but fixed:
  \begin{align}
  \label{EqPfSolHighEnergy}
    \bigl\|(\wh\cP(\sigma)^{-1}&v)|_{\{r<r_1+3\delta'\}}\bigr\|_{H^{-1}} + \sum_{j=0}^s\la\sigma\ra^j\bigl\|(\wh\cP(\sigma)^{-1}v)|_{\{r>r_1+2\delta'\}}\bigr\|_{H^{s-j}} \nonumber\\
      &\leq C\la\sigma\ra\Bigl(\bigl\|v|_{\{r<r_1+2\delta'\}}\bigr\|_{H^{-1}} + \sum_{j=0}^s\la\sigma\ra^j\bigl\|v|_{\{r>r_1+\delta'\}}\bigr\|_{H^{s-j}}\Bigr),
  \end{align}
  for all $v\in\CIc(X^\circ)$. The estimates for $\sigma$ in $\Im\sigma\leq 0$ use the normally hyperbolic nature of the trapping, which for Reissner--Nordstr\"om is easy to establish using spherical symmetry (see \cite{HintzVasyCauchyHorizon} for the case of Reissner--Nordstr\"om--de Sitter, which is completely analogous in this respect), and which for Kerr spacetimes was established by Wunsch and Zworski \cite{WunschZworskiNormHypResolvent} for small angular momenta and by Dyatlov \cite{DyatlovWaveAsymptotics} for the full subextremal range. The restriction to small $\Im\sigma$ in $\Im\sigma>0$ can be relaxed if one replaces the $H^{-1}$ norm in \eqref{EqPfSolHighEnergy} by a weaker Sobolev norm on the left and a stronger one on the right; the precise regularity one can put in, depending on $\Im\sigma$, is dictated by the dynamics near the radial points for the operator $\wh\cP(\sigma)$ at the horizons. The regularity at the horizons will be discussed in \S\S\ref{SubsecPfRad} and \ref{SecRadial}; see also Appendix~\ref{SecPfSolAbs}, where the relation of the weak norm one needs to put in on the left to the surface gravity of the Cauchy horizon is given explicitly.

  The estimate \eqref{EqPfSolHighEnergy} is stated using the standard $H^{-1}$ space rather than the more natural semiclassical Sobolev space $H^{-1}_{\la\sigma\ra^{-1}}$ (or, even better, variable order spaces, see \cite[\S2.4]{HintzVasyCauchyHorizon}), since this will make the function spaces later on more manageable, see in particular \eqref{EqPfAsyPaleyWiener1}. To pass from the semiclassical to the standard Sobolev space, we observe that the inequalities $\la h\xi\ra^{-1}\lesssim h^{-1}\la\xi\ra^{-1}\lesssim h^{-1}\la h\xi\ra^{-1}$, $0<h<1$, $\xi\in\R^n$, imply
  \begin{equation}
  \label{EqPfSolSclStdSob}
    \|v\|_{H_h^{-1}} \lesssim h^{-1}\|v\|_{H^{-1}} \lesssim h^{-1}\|v\|_{H_h^{-1}},\quad v\in\CIc(X^\circ).
  \end{equation}
\item \label{ItPfSolSolvability} Given $f_\chi$ as in \eqref{EqPfLocUchiFchi}, and given $\nu<0$, there exists $f'\in\CIc(\cU^\circ)$ (needed to ensure solvability for $\cP$ --- the potential issue being the pseudodifferential complex absorption $\cQ$) such that the equation
\begin{equation}
\label{EqPfSolExt}
  \cP \wt u=\wt f:=f_\chi+f'
\end{equation}
has a solution $\wt u$ with the following properties:
  \begin{enumerate}
  \item $\wt u$ vanishes near $H_I$ (and also near $H_{I,0}$), hence by uniqueness of the forward problem for the wave equation,
  \begin{equation}
  \label{EqPfSolExtAgree}
    \wt u=u_\chi\quad\tn{in }r>r_1;
  \end{equation}
  \item we have $\wt u\in\Hb^{0,\nu}(\Omega)$, i.e.\ $\wt u$ grows at most at an exponential rate relative to $L^2$; moreover,
  \begin{equation}
  \label{EqPfSolExtReg2}
    \wt u|_{\{r>r_1+\delta'\}}\in\Hb^{\infty,\nu}(\Omega\cap\{r>r_1+\delta'\})
  \end{equation}
  for any $\delta'>0$.
  \end{enumerate}
\end{enumerate}

In view of \itref{ItPfSolMeromorphic}, we may pick $\nu<0$ in \itref{ItPfSolSolvability} so small that there are no resonances, i.e.\ poles of $\wh\cP(\sigma)^{-1}$, in the strip $0<\Im\sigma<-\nu$.

The technical property alluded to above is the following:

\begin{lemma}
\label{LemmaPfSolAbsence}
  There exists $\cQ$, satisfying the assumptions listed before Remark~\ref{RmkPfSolNoB}, such that there are no real resonances $\sigma\in\R$ for the backwards problem for $\cP$ in the region $r_0-\delta\leq r\leq r_1+\delta$: That is, for no $\sigma\in\R$ does there exist a distributional mode solution $u(r,\omega)$ of the equation $(\Box_{\wt g}-V-i\cQ)e^{-it_*\sigma}u(r,\omega)=0$ which vanishes near $r=r_0-\delta$ and $r=r_1+\delta$.
\end{lemma}

(Put differently, there are no elements in $\ker\wh\cP(\sigma)$, $\sigma\in\R$, with support in $r_0-\delta<r<r_1+\delta$.) Thus, for \itref{ItPfSolMeromorphic}--\itref{ItPfSolSolvability} above, we may assume that $\cQ$ is such that it satisfies this resonance condition. We defer the proof of this lemma to Appendix~\ref{SecPfSolAbs}.

\subsection{Asymptotic analysis of the extended solution}
\label{SubsecPfAsy}

Given the solution $\wt u$ to the equation \eqref{EqPfSolExt}, we now use \eqref{EqPfSolExtAgree} and the a priori decay and regularity assumption \eqref{EqPfLocL2Decay} on $u_\chi$ to deduce the same decay and some a priori regularity for $\wt u$.

First, we note that $\cP$ is dilation-invariant in $\tau$ (translation-invariant in $t_*$) except in the region of spacetime $t_*\leq t_{*,0}+1$, $r_{\cQ,-}\leq r\leq r_{\cQ,+}$, where the complex absorption $\cQ$ is no longer $t_*$-translation invariant. In order to obtain a dilation-invariant problem near the boundary $\tau=0$ at future infinity, we pick a cutoff $\chi_\tau\in\CI(\R_\tau)$, $\chi_\tau\equiv 1$ for $\tau<e^{-t_{*,0}-2}$ and $\chi_\tau\equiv 0$ for $\tau\geq e^{-t_{*,0}-1}$, and set $u_e=\chi_\tau\wt u$, which solves
\begin{equation}
\label{EqPfAsyNormalOpEqn}
  N(\cP)u_e = f_e,\quad f_e=f_e'+f_e'',\quad f_e':=\chi_\tau\wt f,\quad f_e'':=[N(\cP),\chi_\tau]\wt u.
\end{equation}
Now, $f_e''$ vanishes near $\tau=0$, and therefore lies in $\Hb^{\infty,\infty}$ in $r>r_1+\delta'$, and in $\Hb^{-1,\infty}$ in $r<r_1+\delta'$, for all $\delta'>0$. Taking the Mellin transform in $\tau$ (Fourier transform in $-t_*$), $\wh{u_e}(\sigma)=\int_0^\infty\tau^{-i\sigma}u_e(\tau)\,\frac{d\tau}{\tau}$, we get
\begin{equation}
\label{EqPfAsyInvMellin}
  \wh{u_e}(\sigma)=\wh\cP(\sigma)^{-1}\wh{f_e}(\sigma),\quad \Im\sigma=-\nu.
\end{equation}
Using bounds for $\wh\cP(\sigma)$, and $\wh{f_e}(\sigma)$ for $\Im\sigma\geq 0$, we wish to take the inverse Mellin transform along the contour $\Im\sigma=0$.

We write $\wh{f_e}(\sigma)=\wh{f_e'}(\sigma)+\wh{f_e''}(\sigma)$ and start by considering the second term: With $\sigma=\sigma_1+i\sigma_2$, $\sigma_1,\sigma_2\in\R$, we have, using \eqref{EqPfSolSclStdSob},
\begin{equation}
\label{EqPfAsyPaleyWiener1}
\begin{split}
  \wh{f_e''}(\sigma)|_{\{r<r_1+\delta'\}} &\in C^0\bigl([0,\infty)_{\sigma_2}; \la\sigma\ra H^{\alpha-1/2-\eps}(\R_{\sigma_1}; H^{-1}(\{r<r_1+\delta'\}))\bigr), \\
  \wh{f_e''}(\sigma)|_{\{r>r_1+\delta'\}} &\in C^0\bigl([0,\infty)_{\sigma_2}; \la\sigma\ra^{-j}H^{\alpha-1/2-\eps}(\R_{\sigma_1};H^{N-j}(\{r>r_1+\delta'\}))\bigr)
\end{split}
\end{equation}
for all $\delta'>0$ and $N\in\Z_{\geq 0}$, $0\leq j\leq N$; moreover, $\wh{f_e''}(\sigma)$ is holomorphic in $\Im\sigma>0$ with uniform bounds
\begin{equation}
\label{EqPfAsyPaleyWiener2}
  \la\sigma\ra^{-1}\|\wh{f_e''}(\sigma)\|_{H^{-1}(X)},\ \|\wh{f_e''}(\sigma)\|_{H^N(\{r>r_1+\delta'\})} \leq C_{N,\delta',\eta},
\end{equation}
for $N\in\Z_{\geq 0}$, $\Im\sigma>\eta>0$, due to the forward support property $\supp f_e''\subset\{t_*\geq 0\}$. (Since $f_e''$ is exponentially decaying, one could make much stronger statements, which however are of no use for us here.) Conversely, \eqref{EqPfAsyPaleyWiener1} and \eqref{EqPfAsyPaleyWiener2} together imply $f_e''\in t_*^{-\alpha+1/2+\eps}H^{-2}(\{r<r_1+\delta'\})\cap t_*^{-\alpha+1/2+\eps}H^\infty(\{r>r_1+\delta'\})$, $\delta'>0$, as well as the support property. (The weight $\la\sigma\ra$ in the first line of \eqref{EqPfAsyPaleyWiener1} amounts to the loss of one $D_{t_*}$ derivative in addition to the $H^{-1}$ in the spatial variables, hence we get a spacetime $H^{-2}$ membership.)

On the other hand, $f_e'$ vanishes near $r=r_1$, hence a forteriori \eqref{EqPfAsyPaleyWiener1}--\eqref{EqPfAsyPaleyWiener2} hold as well for $f_e'$ in place of $f_e''$, and hence for $f_e=f_e'+f_e''$.

Now, when shifting the contour of the inverse Mellin transform of \eqref{EqPfAsyInvMellin} down to the real line, the only potential issue are real resonances of $\cP$. However, since the metric $\wt g$ is a perturbation of a spherically symmetric one, it is easy to see that there are none. Indeed, due to Lemma~\ref{LemmaPfSolAbsence}, every real resonance must be visible in $(r_2-\delta,r_3+\delta)$ already, i.e.\ the rank of the resonance is equal to the rank of the pole of $R\circ\wh\cP(\sigma)^{-1}\circ E$, where $E\colon\CIc((r_2-\delta,r_3+\delta))\hookrightarrow \CI(X)$ denotes extension by $0$, and $R\colon\CI(X)\to\CI((r_2-\delta,r_3+\delta))$ is the restriction map. But then the integration by parts arguments of \cite[\S4]{HintzVasyKdsFormResonances}, which apply in the spherically symmetric case, specifically to Schwarzschild--de Sitter metrics and modified Reissner--Nordstr\"om metrics $\wt g$, and the non-negativity of the potential $V\not\equiv 0$ imply that the latter operator family has no poles in $\Im\sigma\geq 0$; this persists for small perturbations of the metric in the region $r_2-\delta\leq r\leq r_3+\delta$. Thus, $\wh\cP(\sigma)^{-1}$ is holomorphic in a strip around the real axis.

Let us now consider the inverse Mellin transform
\[
  u_e(\tau) = \frac{1}{2\pi}\int_{\Im\sigma=\eta} \tau^{i\sigma} \wh\cP(\sigma)^{-1}\wh{f_e}(\sigma)\,d\sigma,
\]
which is independent of $\eta>0$, and by continuity, using \eqref{EqPfAsyPaleyWiener1}, equal to its value (which are functions on level sets of $\tau$ in the spacetime) at $\eta=0$. Decay of this function in $t_*=-\log\tau$ is a consequence of \eqref{EqPfAsyPaleyWiener1} and the mapping property \eqref{EqPfSolHighEnergy} for $\wh\cP(\sigma)^{-1}$, which imply that $\wh\cP(\sigma)^{-1}\wh{f_e}(\sigma)$ lies in the spaces on the right hand side of \eqref{EqPfAsyPaleyWiener1} with $\la\sigma\ra$ and $\la\sigma\ra^{-j}$ replaced by $\la\sigma\ra^2$ and $\la\sigma\ra^{-j+1}$ due to the factor $\la\sigma\ra$ on the right hand side of \eqref{EqPfSolHighEnergy}; the only slightly subtle part is the persistence of the $H^{\alpha-1/2-\eps}$ nature of the function space in $\sigma_1$, which however follows by interpolation from the persistence of $H^0$ regularity, i.e.\ if we had $\alpha-1/2-\eps=0$ --- this is automatic simply from the operator bounds on $\wh\cP(\sigma)^{-1}$ ---, and the persistence of $H^k$ regularity for $k\in\Z_{\geq 1}$, which for $k=1$ follows from
\[
  \pa_\sigma(\wh\cP(\sigma)^{-1}\wh{f_e}(\sigma))=\pa_\sigma\wh\cP(\sigma)^{-1}\wh{f_e}(\sigma) + \wh\cP(\sigma)^{-1}\pa_\sigma\wh{f_e}(\sigma)
\]
and the fact that the holomorphicity of $\wh\cP(\sigma)^{-1}$ in a strip around the reals implies that $\pa_\sigma\wh\cP(\sigma)^{-1}$, given by the Cauchy integral formula in a disk of fixed radius around $\sigma$, satisfies the same operator bounds as $\wh\cP(\sigma)^{-1}$; similarly for larger values of $k$. The estimate \eqref{EqPfSolHighEnergy} thus gives
\begin{equation}
\label{EqPfAsyDecayLowReg}
  u_e\in t_*^{-\alpha-1/2-\eps}H^{-3}(\{r<r_1+\delta'\})\cap t_*^{-\alpha-1/2-\eps}H^\infty(\{r>r_1+\delta'\}),\quad \delta'>0,
\end{equation}
(See the discussion in the paragraph following \eqref{EqPfAsyPaleyWiener2}. The additional factor of $\la\sigma\ra$ on the right hand side of \eqref{EqPfSolHighEnergy} leads to the additional loss of one derivative in the present, crude analysis.)

We stress again that the function spaces here are \emph{spacetime} Sobolev spaces.

\begin{rmk}
\label{RmkPfAsySmallA}
  The absence of real resonances of the inverse family $\cP$ on the real line is where we used that the spacetime $(\wt M,\wt g)$ under consideration only deviates from spherical symmetry by a small perturbation. We expect that a modification of our construction guarantees this for the full subextremal Kerr family; however, we do not pursue this further here.
\end{rmk}

\subsection{Improved regularity at the Cauchy horizon}
\label{SubsecPfRad}

The passage to the normal operator family, which involved a simple cutoff, and the crude mapping properties we used above give very weak information on $u_e$ in terms of regularity; however, the polynomial weight is optimal relative to the $L^2$ decay which we started with in \S\ref{SubsecPfLoc}.

We now use that $u_e$ solves the equation $N(\cP)u_e=f_e$ and the crude regularity information in \eqref{EqPfAsyDecayLowReg} together with precise radial point estimates to re-gain regularity. The setup of the radial point estimate is detailed in \S\ref{SubsecRadial}, see also Figure~\ref{FigPfSolFlow}, and the proof that the null-geodesic flow on our spacetime $(\wt M,\wt g)$ satisfies the assumptions of the radial point propagation result follows from a simple calculation, which is presented in detail in \cite[\S\S2.2, 3.2]{HintzVasyCauchyHorizon}. If one applied the below-threshold result \cite[Proposition~2.1]{HintzVasySemilinear}, recalled here in Proposition~\ref{PropRadialRecall} \itref{ItRadialRecallOut}, or rather a version allowing for logarithmic weights in $\tau$ (recall that $t_*=\log\tau^{-1}$), one would obtain $u_e\in t_*^{-\alpha-1/2-\eps}H^{1/2-\delta}$ near $r=r_1$ for any $\delta>0$. However, in $1$ dimension (only the radial direction will be subtle, in $t_0$ and the spherical variables, $u_e$ will be shown to be smooth), $H^{1/2-\delta}$ barely fails to embed into $L^\infty$, so we need a more delicate estimate.

We therefore apply Theorem~\ref{ThmRadialLogOut} (with the order of the operator being $m=2$, the weight $r=0$, so the threshold regularity is $s=1/2$), where we take $\ell+1/2=\alpha-1/2+\eps$ for the logarithmic weight; notice that the a priori logarithmic weight in a punctured neighborhood of the radial set at $\tau=0,r=r_1$ is $\ell+1/2$. The (logarithmic) regularity requirement in the punctured neighborhood holds by a wide margin, due to the second part of \eqref{EqPfAsyDecayLowReg} in $r>r_1$, and by elliptic regularity on the elliptic set of $\cQ$ and real principal type propagation in $r<r_1$, and due to the fact that $f_e$ vanishes near $\tau=0$, $r=r_1$. Since the logarithmic weight $(\log\tau^{-1})^{-\ell}$ is equal to $t_*^{-\ell}$, we thus conclude that
\begin{equation}
\label{EqPfRadRegExtremal}
  u_e \in \bigl(t_*^{-\alpha+1-\eps}H^{1/2}(\{|r-r_1|<\delta'\})\bigr) \cap H^{1/2+(\alpha-1+\eps)\log}(\{|r-r_1|<\delta'\})
\end{equation}
for any $\eps>0$, and $\delta'>0$, and in fact the interpolated version
\begin{equation}
\label{EqPfRadReg}
  u_e \in t_*^{-\alpha+3/2-\eps}H^{1/2+(1/2+\eps)\log}(\{|r-r_1|<\delta'\}),\quad \eps>0.
\end{equation}
Now, $H^{1/2+(1/2+\eps)\log}(\R)\hookrightarrow L^\infty(\R)$ in view of $\la\xi\ra^{-1}\la\log\la\xi\ra\ra^{-1-2\eps}\in L^1$; thus, in order to conclude \eqref{EqPfMainResBdd} in Theorem~\ref{ThmPfMain}, it remains to establish the iterative regularity \eqref{EqPfMainResReg} under vector fields in the set $\cM$. (Note that we only prove the theorem for $u_e$, rather than $u$ itself, but $u_e\equiv u$ near $r=r_1$ in $\tau\leq e^{-t_{*,0}-2}$ due to the properties of $\chi_\tau$ in \eqref{EqPfAsyNormalOpEqn}; thus, considering the extended problem on a slightly larger region instead, we obtain regularity for the extended solution in $\tau\leq e^{-t_{*,0}}$, which gives the desired conclusion for $u$ upon restriction.)

To do this, we use the Killing and hidden symmetries of Reissner--Nordstr\"om and Kerr: Firstly, since $D_{t_*}$ commutes with $\Box_g$, we obtain \eqref{EqPfRadReg} for any number of $t_*$-derivatives of $u_e$. On the Reissner--Nordstr\"om spacetime, we may commute the spherical Laplacian $\Delta_{\Sph^2}$ through the equation and thus obtain smoothness in the angular variables as well, finishing the proof in this case, since this implies that $u_e$ lies in the space
\begin{equation}
\label{EqPfRadRegIter}
  u_e\in t_*^{-\alpha+3/2-\eps}H^{1/2+(1/2+\eps)\log}\bigl((r_1-\delta',r_1+\delta');H^N(\R_{t_*}\times\Sph^2)\bigr)\quad \forall N,
\end{equation}
and the $H^N$ space is contained in $L^\infty$ for $N>3/2$.

In the Kerr case, we need to use the \emph{(modified) Carter operator} $\cC\in\Diffb^2(\wt M)$ to gain regularity in the angular variables: This operator is given by
\begin{align*}
  \cC &= \frac{1}{\sin\theta}D_\theta\kappa\sin\theta D_\theta + \frac{(1+\gamma)^2}{\kappa\sin^2\theta}D_{\phi_*}^2 \\
    &\qquad + \frac{2a(1+\gamma)^2}{\kappa}D_{t_*}D_{\phi_*} + \frac{(1+\gamma)^2 a^2\sin^2\theta}{\kappa}D_{t_*}^2
\end{align*}
in the notation introduced at and after \eqref{EqPfKdSMetric}. Now, $D_{t_*}$ and $D_{\phi_*}$ commute with $\rho^2\Box_g$, and moreover the sum of the first two terms of $\cC$ is an elliptic operator on $\Sph^2$; thus, commuting $\cC$ through the equation $\rho^2\Box_g u=\rho^2 f$ in $r>r_1-2\delta$, we see that $u$ is indeed smooth in $t_*$ and the angular variables. We therefore conclude that \eqref{EqPfRadRegIter} holds in the Kerr case as well, finishing the proof of \eqref{EqPfMainResReg}.

For the last part of Theorem~\ref{ThmPfMain}, where we now assume \eqref{EqPfMainAssReg2}, i.e.\ derivatives of $D_{t_*}u_e$ in the exterior region lie in $t_*^{-\alpha-1}L^\infty$, the previous arguments yield
\begin{gather*}
  u_e\in t_*^{-\alpha+3/2-\eps}L^\infty\bigl(\R_{t_*}\times(r_1-\delta',r_1+\delta')\times\Sph^2\bigr), \\
  D_{t_*}u_e\in t_*^{-\alpha+1/2-\eps}L^2\bigl(\R_{t_*},L^\infty((r_1-\delta',r_1+\delta')\times\Sph^2)\bigr);
\end{gather*}
the inequality $\|v\|_{L^\infty}\leq t_*^{-1/2}\|t_*D_{t_*}v\|_{L^2}$ for $v\in\CIc((1,\infty)_{t_*})$ thus yields $u_e\in t_*^{-\alpha+1-\eps}L^\infty$, as claimed.

\begin{rmk}
\label{RmkPfRadAmountDecay}
  The conclusion \eqref{EqPfRadRegExtremal} only requires $\alpha>1$; however, in order to get an $L^\infty$ bound in the radial variables, even without polynomial decay, we need $\alpha>3/2$, even though in the last case of Theorem~\ref{ThmPfMain}, we only lose $1+\eps$ in the polynomial decay rate.
\end{rmk}

\section{Logarithmic improvements in radial point estimates}
\label{SecRadial}

This section is the technical heart of the paper. To prepare the logarithmic regularity improvement in the microlocal blue-shift (radial point) estimate at the Cauchy horizon, we first discuss the relevant function spaces and spaces of operators in \S\ref{SubsecSobolev} before turning to the proof of the estimate in \S\ref{SubsecRadial}.

We remark that results similar to what we prove here can be obtained in the simpler setting of manifolds without boundary; we leave the details to the interested reader.

\subsection{Sobolev spaces with logarithmic derivatives and weights}
\label{SubsecSobolev}

In this section, we discuss b-\psdo{}s, on manifolds with boundary, with logarithmic derivatives and logarithmic weights at the boundary. Since $\log\la\xi\ra$, $\xi\in\R^n$, is a symbol of order $\eps$ for all $\eps>0$, we will be able to view the `logarithmic' b-calculus as a special case of the standard b-calculus, with more precise composition properties due to the full symbolic expansion of the standard calculus.

In terms of the spaces
\[
  H^{s+\ell\log}(\R^n) := \bigl\{ u\in\sS'(\R^n)\colon \la\xi\ra^s\la\log\la\xi\ra\ra^\ell\wh u(\xi)\in L^2(\R^n_\xi)\bigr\},
\]
the logarithmic versions of weighted b-Sobolev spaces on the half-space $\Rnhalfc=[0,\infty)_x\times\R^{n-1}_y$ are defined as follows:

\begin{definition}
\label{DefSobolevRnhalf}
  For $s,\ell,r,j\in\R$, the space $\Hb^{s+\ell\log,r+j\log}(\Rnhalfc)$ consists of all distributions $u\in\CmI(\Rnhalfc)$ of the form
  \[
    u =  x^r \log^{-j}(\wt x^{-1}) u_0,\quad u_0\in\Hb^{s+\ell\log}(\Rnhalfc),
  \]
  where $\Hb^{s+\ell\log}(\Rnhalfc)=F^*(H^{s+\ell\log}(\R^n))$, $F(x,y)=(-\log x,y)$, and $\wt x$ is a smoothed out version of the function $\min(x,1/2)$, for instance $\wt x=x$ for $0\leq x\leq 1/4$, $\wt x=1/2$ for $x\geq 1/2$, and $1/4\leq\wt x\leq 1/2$ for $1/4\leq x\leq 1/2$.
\end{definition}

Thus, for $j=0$, one obtains the usual weighted space $\Hb^{s+\ell\log,r}=x^r\Hb^{s+\ell\log}$, while non-zero $j$ allows one to distinguish logarithmic orders of growth at the boundary $x=0$. The space $\Hb^{s+\ell\log,r+j\log}$ becomes smaller when any one of the quantities $s,\ell,r,j$ is increased.

In order to handle logarithmic weights, we consider b-\psdo{}s $A\in\Psib^s(\Rnhalfc)$ on $\Rnhalfc$ (with compactly supported Schwartz kernel vanishing to infinite order at the two side faces of the b-double space $\Rnhalfc\times_\bop\Rnhalfc$, as we will always implicitly assume) with merely \emph{conormal} \emph{left-reduced symbols} $a$, which are functions on $\Tb^*\Rnhalfc\cong\Rnhalfc\times\R^n$ (compactly supported in the base variables) satisfying the estimates
\[
  |(x\pa_x)^q \pa_y^\beta \pa_\xi^k \pa_\eta^\alpha a(x,y,\xi,\eta)| \leq C_{k\alpha q\beta}\la(\xi,\eta)\ra^{s-(k+|\alpha|)}\quad\forall q,k\geq 0, \forall \alpha,\beta,
\]
Under the diffeomorphism $F$ in Definition~\ref{DefSobolevRnhalf}, writing $F(x,y)\equiv(t,y)$, these are the uniform symbols $a(t,y,\xi,\eta)\in S^s_\infty(\R^n_{t,y};\R^n_{\xi,\eta})$ of \cite[\S2.1]{MelroseIML}. In fact, as in \cite[\S2.2]{MelroseIML}, one can consider more general symbols, say with support in $x,x'\leq 1/2$, satisfying
\begin{equation}
\label{EqSobolevTwoSided}
\begin{split}
  |(x\pa_x)^q\pa_y^\beta& (x'\pa_{x'})^{q'}\pa_{y'}^{\beta'} \pa_\xi^k \pa_\eta^\alpha a(x,y,\xi,\eta,x',y')| \\
    &\leq C_{k\alpha q\beta q'\beta'}\la\log x-\log x'\ra^w\la(\xi,\eta)\ra^{s-(k+|\alpha|)}
\end{split}
\end{equation}
for all $q,q',k\geq 0$ and all multi-indices $\alpha,\beta,\beta'$; here $w\in\R$ is an off-diagonal weight. Quantizations of such symbols can always be written as quantizations of left-reduced symbols (in particular, the weight $w$ is irrelevant!), so the increased generality in the symbol class does not enlarge the class of b-\psdo{}s with conormal coefficients; see \cite[\S2.4]{MelroseIML}. The point in allowing \eqref{EqSobolevTwoSided} is that it allows us to show:

\begin{lemma}
\label{LemmaSobolevConj}
  If $A\in\Psib^s(\Rnhalfc)$ has compactly supported Schwartz kernel with support in $x,x'\leq 1/2$, then $A_j:=\log^j(x^{-1})\circ A\circ \log^{-j}(x^{-1}) \in \Psib^s(\Rnhalfc)$ for any $j\in\R$.
\end{lemma}
\begin{proof}
  If $a(x,y,\xi,\eta)$ is the left-reduced symbol of $A$, then
  \[
    a_j(x,y,\xi,\eta,x',y') = \Bigl(\frac{\log x}{\log x'}\Bigr)^j a(x,y,\xi,\eta)
  \]
  is the full symbol of $A_j$. The proof is completed by noting that $|(a/b)^j-1|\leq C_\delta\la a-b\ra^{|j|}$ for $a,b\geq\delta>0$, which follows for $j\geq 0$ by writing $|(a/b)^j-1|=|(1+(a-b)/b)^j-1|$ and separating the cases $|a-b|<|b|/2$ and $|a-b|\geq|b|/2$; the inequality for $j\leq 0$ follows from the one for $j\geq 0$ by replacing $(a,b,j)$ by $(b,a,-j)$.
\end{proof}

We will also need to consider b-\psdo{}s allowing for `logarithmic' derivatives; this is accomplished by adding to the right hand side of \eqref{EqSobolevTwoSided} the factor $\la\log\ra^\ell(\xi,\eta)$, $\ell\in\R$, yielding the symbol class $S^{s+\ell\log}$; quantizations of such symbols give rise to the space $\Psib^{s+\ell\log}(\Rnhalfc)$. We can transfer such spaces of b-\psdo{}s to a manifold with boundary $M$: Indeed, since for all $\eps>0$, the inclusion $S^{s+\ell\log}\subset S^{s+\eps}$ holds, we have $\Psib^{s+\ell\log}\subset\Psib^{s+\eps}$. Furthermore, $S^{-\infty+\ell\log}=S^{-\infty}$ for any $\ell\in\R$. Hence, by the asymptotic formula for the full symbol of an element of $\Psib^{s+\eps}(\Rnhalfc)$ under a diffeomorphism which is the identity outside a compact set, the class $\Psib^{s+\ell\log}(\Rnhalfc)$ is invariant under such diffeomorphisms as well. Thus, one obtains the $*$-algebra
\[
  \bigcup_{s,\ell}\Psib^{s+\ell\log}(M)
\]
invariantly on the manifold $M$ with boundary. We can then use this calculus to define weighted b-Sobolev spaces
\[
  \Hbloc^{s+\ell\log,r+j\log}(M) = x^r\log^{-j}(x^{-1})\Hbloc^{s+\ell\log}(M),
\]
$x$ a boundary defining function of $M$ with $x\leq 1/2$ everywhere, where $\Hbloc^{s+\ell\log}(M)$ consists of all distributions $u$ on $M$ for which $A u\in L_\bop^2(M)$ for every compactly supported $A\in\Psib^{s+\ell\log}(M)$; here $L_\bop^2(M)=L^2(M,x^{-1}\nu)$ for a non-vanishing smooth density $\nu$.

Furthermore, for $u\in\Hbloc^{-\infty,r+j\log}(M)$, we can define $\WFb^{s+\ell\log,r'+j'\log}(u)$ for $s,\ell\in\R$ and $r'\leq r$, $j'\leq j$ to consist of all $\zeta\in\Sb^*M$ for which there exists $A\in\Psib^0(M)$, elliptic at $\zeta$, such that $Au\in\Hbloc^{s+\ell\log,r'+j'\log}(M)$. Notice here that $u$ lies in $\Hbloc^{-\infty,r'+j'\log}(M)$ a priori by the conditions on $r,j,r',j'$, hence these conditions ensure that the wave front set indeed measures global properties: If $\WFb^{s+\ell\log,r'+j'\log}(u)=\emptyset$ for such $u$, then $u\in\Hbloc^{s+\ell\log,r'+j'\log}(M)$. (This fact relies on elliptic regularity on these b-Sobolev spaces, which holds by the usual symbolic argument.)

We end this section by discussing a special class of operators which will naturally appear in the logarithmically improved radial point estimates in \S\ref{SubsecRadial}. To wit, let $z=(x,y)\in[0,1/2)\times\R^{n-1}_y$ and $\zeta\in\R^n$ denote coordinates on $\Tb^*M$ near $\pa M$; let $\rho=|\zeta|^{-1}$ for $|\zeta|\geq 4$ (where we use any fixed norm on the fibers of $\Tb^*M$ induced by a Riemannian b-metric on $M$), with $0<\rho<1/2$ everywhere on $\Tb^*M$, and consider the symbol
\[
  a(z,\zeta) = \log^\ell(x^{-v}\rho^{-w}),\quad \ell>0,\ v,w>0.
\]
Since
\begin{equation}
\label{EqSobolevDoubleSymbolEst}
  a(z,\zeta) = \log^\ell(x^{-1}) \Bigl(v + w\frac{\log\rho^{-1}}{\log x^{-1}}\Bigr)^\ell
\end{equation}
lies in the symbol class $\log^\ell(x^{-1}) S^{\ell\log}$, we can quantize it and obtain an operator $A$ in the class $\log^\ell(x^{-1})\circ\Psib^{\ell\log}$, for which we have good composition and mapping properties. Then:

\begin{lemma}
\label{LemmaSobolevInterpolation}
  Let $\ell,v,w>0$ and $a(z,\zeta)$, $A$ be as above.
  \begin{enumerate}
  \item \label{ItSobolevInterpolation1} Suppose
    \[
      u\in\Hbloc^{-\infty,r+\ell\log},\quad A u\in\Hbloc^{s,r}.
    \]
    Then
    \begin{equation}
    \label{EqSobolevInterpolation}
      u\in\Hbloc^{s+\alpha\ell\log,r+(1-\alpha)\ell\log},\quad\alpha\in[0,1].
    \end{equation}
  \item \label{ItSobolevInterpolation2} Conversely,
    \begin{equation}
    \label{EqSobolevInterpolation2}
      u\in\Hbloc^{s+\ell\log,r} \cap \Hbloc^{s,r+\ell\log} \implies A u\in\Hbloc^{s,r}.
    \end{equation}
    (This in turn gives \eqref{EqSobolevInterpolation}.)
  \end{enumerate}
\end{lemma}

That is, we can interpolate logarithmic weights and regularity.

\begin{proof}[Proof of Lemma~\ref{LemmaSobolevInterpolation}]
  Part \itref{ItSobolevInterpolation1} follows from the fact that
  \begin{equation}
  \label{EqSobolevInterpolationPf}
    \log^{(1-\alpha)\ell}(x^{-1}) \log^{\alpha\ell}(\rho^{-1}) = b_\alpha a
  \end{equation}
  with a symbol $b_\alpha\in S^0$; this in turn is a consequence of the estimate
  \[
    \Bigl(\frac{\wt x^{1-\alpha}\wt\rho^\alpha}{v\wt x+w\wt\rho}\Bigr)^\ell \leq \Bigl(\frac{(1-\alpha)\wt x+\alpha\wt\rho}{v\wt x+w\wt\rho}\Bigr)^\ell \leq C_\ell
  \]
  for $\wt x=\log x^{-1}$ and $\wt\rho=\log\rho^{-1}$, and similar estimates for symbolic derivatives.
  
  To prove part \itref{ItSobolevInterpolation2}, we note that
  \[
    (v\log x^{-1}+w\log\rho^{-1})^\ell = \log^{\ell}(\rho^{-1})b_1 + \log^{\ell}(x^{-1})b_2
  \]
  for symbols $b_1,b_2\in S^0$; explicitly, let $\chi\in\CI(\R)$ be identically $0$ in $(-\infty,1/2]$ and identically $1$ in $[3/2,\infty)$, then we can take
  \[
    b_1 = \phi\Bigl(\frac{\log\rho^{-1}}{\log x^{-1}}\Bigr)\Bigl(1+\frac{\log x^{-1}}{\log\rho^{-1}}\Bigr)^\ell, \quad  b_2 = \Bigl(1-\phi\Bigl(\frac{\log\rho^{-1}}{\log x^{-1}}\Bigr)\Bigr)\Bigl(1+\frac{\log\rho^{-1}}{\log x^{-1}}\Bigr)^\ell.\qedhere
  \]
\end{proof}

\subsection{Setup and proof of the radial point estimate}
\label{SubsecRadial}

Let $M$ be a manifold with boundary $X$, and let $\tau$ be a boundary defining function. Let $\cP\in\Psib^m(M)$ be a b-\psdo{}\ on $M$ with homogeneous principal symbol $p\in S^m(\Tb^*M)$. We denote by $\Sigma=p^{-1}(0)\subset\Tb^*M\setminus o$ the characteristic set of $\cP$, let $\wh\rho\in\CI(\Tb^*M)$ be homogeneous of degree $-1$ away from the origin and non-vanishing everywhere, and define the rescaled Hamilton vector field
\[
  \rham_p = \wh\rho^{m-1}\ham_p,
\]
which is a b-vector field on the radial compactification $\rcTb^*M$ of the b-cotangent bundle, i.e.\ it is tangent to $\rcTb^*_X M$ and $\Sb^*M$. (See \cite[\S2]{HintzVasySemilinear} for details.) With $\Lambda\subset\Sigma\cap\Tb^*_X M\setminus o$ denoting a conic submanifold of $\Sigma$ invariant under the $\ham_p$ flow, and $L:=\pa\Lambda\subset\Sb^*_X M$, we assume that $L$ is a saddle point for the $\rham_p$ flow, more precisely a source within $\rcTb^*_X M$ (in particular also in the fiber radial direction), with a stable manifold $\cL$ transversal to $X$:
\begin{enumerate}
  \item The characteristic set $\Sigma$ is a $\CI$ codimension $1$ submanifold of $\Tb^*M\setminus o$ and transversal to $\Tb^*_X M\setminus o$;
  \item $L=\cL\cap\Sb^*_X M$, where $\cL\subset\pa\Sigma\subset\Sb^*M$ is a smooth submanifold, transversal to $\Sb^*_X M$ and invariant under the $\rham_p$ flow;
  \item we have
    \[
    \begin{split}
      \rham_p\wh\rho&=\beta_0\wh\rho,\quad \beta_0\in\CI(S^*X)\tn{ near }L,\ \beta_0>0\tn{ at }L, \\
      \rham_p\tau&=-\beta_0\wt\beta\tau,\quad \wt\beta\in\CI(S^*X)\tn{ near }L,\ \wt\beta>0\tn{ at }L;
    \end{split}
    \]
    For simplicity, let us assume
    \[
      \wt\beta \equiv \beta\tn{ is constant};
    \]
    this is satisfied in our applications.
  \item there exists a quadratic defining function $\rho_0$ of $L$ within the characteristic set $\pa\Sigma\subset\Sb_X^*M$ over $X$ such that
  \begin{equation}
  \label{EqRadialQuadr}
    \rham_p\rho_0\geq\beta_q\rho_0
  \end{equation}
  in a neighborhood of $L$ within $\pa\Sigma$, where $\beta_q\in\CI(\Sb^*X)$, $\beta_q>0$ at $L$;
  \item we have $\cP-\cP^*\in\Psib^{m-2}(M)$, i.e.\ $\cP$ is symmetric modulo an operator \emph{two} orders lower.
\end{enumerate}

We refer back to Figure~\ref{FigPfSolFlow} for an illustration. We recall the following result on the propagation of singularities at such (normal) saddles of the null-bicharacteristic flow:
\begin{prop}
\label{PropRadialRecall}
  \cite[Proposition~2.1]{HintzVasySemilinear}. Let $s,r\in\R$.
  \begin{enumerate}
    \item \label{ItRadialRecallInto} Suppose $s>s_0>(m-1)/2+\beta r$. There exist $B,B',E,G\in\Psib^0(M)$, microlocalized in any fixed neighborhood of $L$, with $B,B'$ and $G$ elliptic at $L$, and with $\WFb'(E)\cap\Sb^*_X M=\emptyset$, such that the estimate
      \[
        \|B u\|_{\Hb^{s,r}} \leq C(\|G \cP u\|_{\Hb^{s-m+1,r}} + \|E u\|_{\Hb^{s,r}} + \|B' u\|_{\Hb^{s_0,r}} + \|u\|_{\Hb^{-N,r}}).
      \]
      holds for all $u$ for which the right hand side is finite. Here, $N\in\R$ is arbitrary but fixed.
    \item \label{ItRadialRecallOut} Suppose $s<(m-1)/2+\beta r$. There exist $B,E,G\in\Psib^0(M)$, microlocalized in any fixed neighborhood of $L$, with $B$ and $G$ elliptic at $L$, and with $\WFb'(E)\cap\cL=\emptyset$, such that the estimate
      \[
        \|B u\|_{\Hb^{s,r}} \leq C(\|G \cP u\|_{\Hb^{s-m+1,r}} + \|E u\|_{\Hb^{s,r}} + \|u\|_{\Hb^{-N,r}}).
      \]
      holds for any fixed $N\in\R$.
  \end{enumerate}
\end{prop}

Thus, in \itref{ItRadialRecallInto}, we can propagate $\Hb^{s,r}$ regularity from $\tau>0$ into $L$, assuming a priori $\Hb^{s_0,r}$ regularity at $L$. In both cases in the proposition, we do not gain (or lose) decay: the weight $r$ in the function spaces is the same for a priori assumptions and conclusions. See Figure~\ref{FigRadialAway} for an illustration of \itref{ItRadialRecallOut}.

Both estimates in Proposition~\ref{PropRadialRecall} hold true if one adds logarithmic regularity and logarithmic weights, i.e.\ if one replaces $s$ and $r$ by $s+\ell\log$ and $r+j\log$, respectively, with $\ell,j\in\R$.

\begin{figure}[!ht]
  \centering
  \includegraphics{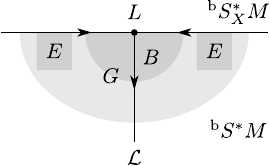}
  \caption{Setup of the radial point estimate below the threshold regularity, propagating regularity within $\Sb^*_X M$ from $\WFb'(E)$ into $\WFb'(B)$ and out into the interior of $\Sb^*M$.}
\label{FigRadialAway}
\end{figure}

\begin{rmk}
  One can allow $\cP-\cP^*\in\Psib^{m-1}(M)$ to be non-trivial with
  \[
    \frac{1}{2i}\sigma_{\bop,m-1}(\cP-\cP^*) = \beta_0\wh\beta\wh\rho^{-(m-1)}
  \]
  which shifts the threshold regularity by $\wh\beta$ if the latter is constant. In general, the threshold regularity is $(m-1)/2+\sup_L(\wt\beta r+\wh\beta)$ for part \itref{ItRadialRecallInto} and $(m-1)/2+\inf_L(\wt\beta r+\wh\beta)$ for part \itref{ItRadialRecallOut}.
\end{rmk}

In the borderline case, let us consider the propagation from the boundary through $L$ into the interior of $M$, where we will refine part \itref{ItRadialRecallOut} of the above proposition, propagating regularity through the saddle point into the interior of $M$:

\begin{thm}
\label{ThmRadialLogOut}
  Let $r\in\R$, and let $s=(m-1)/2+\beta r$ be the borderline regularity in the notation of Proposition~\ref{PropRadialRecall}, and let $\ell\geq 0$. There exist $B,E,G\in\Psib^0(M)$, microlocalized in any fixed neighborhood of $L$, with $B$ and $G$ elliptic at $L$, and with $\WFb'(E)\cap\cL=\emptyset$, such that for all $\alpha\in[0,1]$, the estimate
  \begin{align*}
    \|B&u\|_{\Hb^{s+\alpha\ell\log,r+(1-\alpha)\ell\log}} \\
      &\leq C\bigl(\|G \cP u\|_{\Hb^{s-m+1+(\ell+1)\log,r}}+\|G \cP u\|_{\Hb^{s-m+1,r+(\ell+1)\log}} \\
      &\qquad + \|E u\|_{\Hb^{s+(\ell+1/2)\log,r}}+\|E u\|_{\Hb^{s,r+(\ell+1/2)\log}} + \|u\|_{\Hb^{-N,r+(\ell+1/2)\log}}\bigr)
  \end{align*}
  holds for all $u$ for which the right hand side is finite; here $N\in\R$ are arbitrary but fixed. Moreover, a neighborhood of $L$ intersected with $\tau>0$ is disjoint from $\WF^{s+(\ell+1/2)\log}(u)$.
\end{thm}

Thus, we can gain logarithmic regularity at the cost of giving up logarithmic decay in $\tau$. In the last statement, note that away from the boundary, b-wave front set (with any weight) and ordinary wave front set coincide. The estimate on $u$ at the radial set yields $1$ ordinary derivative less, as usual for hyperbolic propagation, and in addition loses one logarithmic derivative as well as one logarithmic weight relative to $\cP u$, and we lose $1/2$ a logarithm in regularity and decay relative to the regularity of $u$ in a punctured neighborhood of $L$.

\begin{proof}[Proof of Theorem~\ref{ThmRadialLogOut}.]
  We may assume that $\wh\rho<1/2$ everywhere. Let $\sfp:=\wh\rho^m p$ be the rescaled principal symbol, and let $\phi_0,\phi_1,\phi\in\CIc(\R)$ be cutoffs, identically equal to $1$ near $0$, and with $\phi_j'\leq 0$ and $\sqrt{-\phi_j'\phi_j}\in\CI$ on $[0,\infty)$ for $j=0,1$. We moreover assume $\phi_1(\tau)\equiv 0$ for $\tau\geq 1/2$. With regularization in the regularity parameter (i.e.\ the weighting in $\wh\rho$) to be added later, we use the commutant
  \begin{equation}
  \label{EqRadialLogOutComm}
    \check a = \wh\rho^{-s+(m-1)/2}\tau^{-r} \log^{\ell+1/2}(\wh\rho^{-\beta}\tau^{-2}) \phi_0(\rho_0)\phi_1(\tau)\phi(\sfp).
  \end{equation}
  Notice here that $\rham_p(\wh\rho^{-\beta}\tau^{-1})=0$, so the reason for the exponent $2$ of $\tau^{-1}$ is that it gives the correct sign of the Hamilton derivative (any exponent $>1$ would work as well); indeed, we have
  \[
    \rham_p\log^{\ell+1/2}(\wh\rho^{-\beta}\tau^{-2}) = \beta_0\beta(\ell+1/2)\log^{\ell-1/2}(\wh\rho^{-\beta}\tau^{-2}).
  \]
  Moreover, writing
  \[
    \log(\wh\rho^{-\beta}\tau^{-2}) = \beta\log\wh\rho^{-1} + 2\log\tau^{-1},
  \]
  we see that a commutator argument using $\check a$ will allow us to logarithmically improve regularity at the cost of giving up a logarithmic weight in $\tau$; see also the more precise discussion in Lemma~\ref{LemmaSobolevInterpolation}.
  
  Concretely, we compute
  \begin{equation}
  \label{EqRadialCommHam}
  \begin{split}
    \check a&\ham_p\check a = \wh\rho^{-2s}\tau^{-2r}\log^{2\ell}(\wh\rho^{-\beta}\tau^{-2})\phi_0\phi_1\phi \\
      &\times \Bigl( (\ell+1/2)\beta_0\beta\phi_0\phi_1\phi \\
      &\qquad + \log(\wh\rho^{-\beta}\tau^{-2})\bigl((\rham_p\rho_0)\phi_0'\phi_1\phi - \beta_0\beta\tau \phi_1'\phi_0\phi + (\rham_p\sfp)\phi'\phi_0\phi_1\bigr)\Bigr).
  \end{split}
  \end{equation}
  The main term, which is non-zero at $L$, is the first term in the big parenthesis on the right hand side. In view of \eqref{EqRadialQuadr}, the first term in the parenthesis in the last line is non-positive; this is the a priori control term, requiring control of $u$ microlocally in a punctured neighborhood of $L$ in the boundary. The second term is positive --- so has the same sign as the main term --- on $\supp\phi_1'$, which is disjoint from the boundary $\tau=0$, and will allow us to conclude regularity there, while the third term is supported off the characteristic set and is thus irrelevant.

  In order to regularize the argument, we use $\varphi_t(\zeta)=(1+t\zeta)^{-N}$, $N>0$, and compute $\rham_p\varphi_t(\wh\rho^{-1}) = N\beta_0\wt\varphi_t\varphi_t$, $\wt\varphi_t(\zeta)=t\zeta(1+t\zeta)^{-1}$. Letting
  \[
    \check a_t = \varphi_t(\wh\rho^{-1})\check a,\quad t>0,
  \]
  we thus obtain, fixing $0<\eta<(\ell+1/2)\beta_0\beta$,
  \begin{equation}
  \label{EqRadialCommRegHam}
    \check a_t\ham_p\check a_t = \eta\log^{-1}(\wh\rho^{-\beta}\tau^{-2})\wh\rho^{-(m-1)}\check a_t^2 + b_t^2 + b_{1,t}^2 + b_{2,t}^2 - e_t^2 + f_t p,
  \end{equation}
  where
  \begin{align*}
    b_t &= \varphi_t\wh\rho^{-s}\tau^{-r}\log^\ell(\wh\rho^{-\beta}\tau^{-2})\phi_0\phi_1\phi\sqrt{(\ell+1/2)\beta_0\beta-\eta}, \\
    b_{1,t} &= \varphi_t\wh\rho^{-s}\tau^{-r}\log^{\ell+1/2}(\wh\rho^{-\beta}\tau^{-2})\phi_0\phi_1\phi\sqrt{N\beta_0\wt\varphi_t}, \\
    b_{2,t} &= \varphi_t\wh\rho^{-s}\tau^{-r}\log^{\ell+1/2}(\wh\rho^{-\beta}\tau^{-2})\phi_0\phi\sqrt{-\beta_0\beta\tau\phi_1'\phi_1}, \\
    e_t &= \varphi_t\wh\rho^{-s}\tau^{-r}\log^{\ell+1/2}(\wh\rho^{-\beta}\tau^{-2})\phi_1\phi\sqrt{-(\rham_p\rho_0)\phi_0'\phi_0}, \\
    f_t &= \varphi_t^2\wh\rho^{-2s+m}\tau^{-r}\log^{2\ell+1}(\wh\rho^{-\beta}\tau^{-2})(\rham_p\sfp)\phi'\phi\phi_0^2\phi_1^2\sfp^{-1}.
  \end{align*}
  Notice that the regularizer contributes the term $b_{1,t}$, which has the same sign as the main term in the unregularized calculation \eqref{EqRadialCommHam}. Denote by $\check A_t,B_t,B_{1,t},B_{2,t},E_t$ and $F_t$ quantizations of the corresponding lower order symbols; we arrange that $\check A_t$ has uniform wave front set contained in $\supp\check a_0$, is uniformly bounded in the space $\Psib^{s-(m-1)/2+(\ell+1/2)\log} + \log^{\ell+1/2}(\tau^{-1})\Psi^{s-(m-1)/2}$ (cf.\ \eqref{EqSobolevInterpolationPf}), and converges to $\check A_0$ in $\Psib^{s-(m-1)/2+\eps}+\log^{\ell+1/2}(\tau^{-1})\Psi^{s-(m-1)/2+\eps}$ for all $\eps>0$ as $t\to 0$; similarly for $B_t,B_{1,t},B_{2,t},E_t$ and $F_t$. Further, let $\Lambda^\sigma$ and $\Lambda_2^{\rho\log}$ be quantizations of $\wh\rho^{-\sigma}$ and $\log^\rho(\wh\rho^{-\beta}\tau^{-2})$, respectively, where $\sigma,\rho\in\R$. We then obtain
  \begin{equation}
  \label{EqRadialCommOp}
  \begin{split}
    \frac{i}{2}[\cP,\check A_t^*\check A_t] &= \eta(\Lambda_2^{-\half\log}\Lambda^{(m-1)/2}\check A_t)^*(\Lambda_2^{-\half\log}\Lambda^{(m-1)/2}\check A_t) \\
      &\qquad + B_t^*B_t + B_{1,t}^*B_{1,t} + B_{2,t}^*B_{2,t} - E_t^*E_t + F_t \cP + L_t,
  \end{split}
  \end{equation}
  where $L_t\in \tau^{-2r}(\log\tau^{-1})^{2\ell+1}\circ\Psib^{2s-1+(2\ell+1)\log}$ uniformly (use \eqref{EqSobolevDoubleSymbolEst} to see this), with order $\Psib^{2(s-N)-1+\eps}$, $\eps>0$, for $t>0$. When computing the pairing
  \begin{equation}
  \label{EqRadialCommOp2}
    \Re\Big\la\frac{i}{2}[\cP,\check A_t^*\check A_t]u,u\Big\ra = -\frac{1}{2}\Im\la\check A_t u,\check A_t \cP\ra,
  \end{equation}
  regularity is not an issue if we take $N\gg 0$ large enough, so we only need to observe that the weights allow for the integration by parts to be performed; this follows as $u$ is assumed to be in the space $u\in\Hb^{-\infty,r+(\ell+1/2)\log}$. Using \eqref{EqRadialCommOp}, integration by parts to rewrite the left hand side of \eqref{EqRadialCommOp2}, and Cauchy--Schwarz similarly to the proof of Theorem~\ref{ThmRadialLogOut} shows that we obtain a uniform bound on $B_t u,B_{1,t}u,B_{2,t}u$ in $L^2$ provided $\WFb^{s+(\ell+1/2)\log,r}(u)\cup\WFb^{s,r+(\ell+1/2)\log}(u)$ is disjoint from a neighborhood of $\WFb'(E_0)$ and $\WFb^{s-m+1+(\ell+1)\log,r}(\cP u)\cup\WFb^{s-m+1,r+(\ell+1)\log}(\cP u)$ is disjoint from a neighborhood of $\WFb'(\check A_0)$. The fact that we only need to exclude the `extreme' wave front sets, where we either consider maximal logarithmic regularity but no logarithmic weight or no logarithmic regularity but maximal logarithmic weight, to control $E_t u$ and $\check A_t \cP u$ relies Lemma~\ref{LemmaSobolevInterpolation}, see in particular \eqref{EqSobolevInterpolation2}. Lastly, the term $\la L_t u,u\ra$ is controlled by Proposition~\ref{PropRadialRecall} \itref{ItRadialRecallOut} with logarithmic weights, which gives $u\in\Hb^{s-\eps,r+(\ell+1/2)\log}$ at $L$ for all $\eps>0$.

  Now, $b_t$ dominates $b^{(\alpha)}:=\log^{\alpha\ell}(\wh\rho^{-1})\log^{(1-\alpha)\ell}(\tau^{-1})$ for $\alpha\in[0,1]$, see \eqref{EqSobolevInterpolationPf}, in the sense that $b_t/b^{(\alpha)}$ is bounded in $S^0$. On the other hand, $b_{2,t}$ is uniformly elliptic in the class $\Psib^{s+(\ell+1/2)\log}$ near $L$, but away from $\tau=0$. (The seemingly better term $b_{1,t}$ does not give any control in the limit $t\to 0$ because $b_{1,t}\to 0$ pointwise as $t\to 0$.) A standard functional analytic argument using the weak compactness of the unit ball in $L^2$, see e.g.\ \cite[Proposition 7 and \S9]{MelroseEuclideanSpectralTheory}, concludes the proof of the theorem.
\end{proof}

\begin{rmk}
\label{RmkRadialConormal}
  In the dilation-invariant (in $\tau$; translation-invariant in $-\log\tau$) setting, one can prove a result on the propagation of conormal (to $\cL$) regularity relative to b-Sobolev spaces strictly below the threshold regularity using positive commutator arguments similar to the ones used in \cite{HassellMelroseVasySymbolicOrderZero,BaskinVasyWunschRadMink,HaberVasyPropagation}; that is, one can prove iterative regularity of $u$ under the application of $Q\in\Psib^1$ with principal symbol $q$ vanishing on $\Nb^*\cL$. Such arguments rely on approximate commutation properties, e.g.\ of the type that $\ham_p q$ vanishes quadratically on the unstable manifold $\cL$ of $L$. However, at the threshold regularity and working with logarithmic regularity, one encounters the following problem, in the notation of the above proof: The main term of the commutator, using a commutant of the form $\check a q$, arises from differentiating $\check a$ along $\ham_p$, and provides control of $\ell$ logarithmic derivatives of $Q u$, but the term from differentiating $q$ still has $(\ell+1/2)$ logarithmic derivatives and is non-zero in a punctured neighborhood of $L$ within $\Sigma$, and therefore cannot be controlled by the main term.
  
  Requiring stronger commutation properties would easily resolve this issue. However, we use exact commutation properties in \S\ref{SubsecPfRad} for brevity and simplicity.
\end{rmk}

\appendix
\section{Proof of Lemma~\ref{LemmaPfSolAbsence}}
\label{SecPfSolAbs}

In this appendix, we will denote by $\cW$ the operator $\cW=\Box_{\wt g}-V-i\cQ$ in a neighborhood $r_0-\delta\leq r\leq r_1+\delta$. (Thus, $\cW$ is a restriction of the operator called $\cP$ in the rest of the paper.) We recall from \S\ref{SubsecPfSol} that the complex absorbing operator $\cQ$ has Schwartz kernel supported in $\ol\cU\times\ol\cU$, where $\cU=\{r_{\cQ,-}<r<r_{\cQ,+},\tau\leq e^{-t_{*,0}-1}\}$.

We first place the backwards problem for $\cW$ into context: Define $H_{I,0}$ and $H_F$ as in \eqref{EqPfSolWhiteHoleSurf}, and let further
\[
  H_{I,1} := \{ r=r_1+\delta,\ t_*\geq t_{*,0} \};
\]
and let $\Omega_{01}^\circ=\{r_0-\delta\leq r\leq r_1+\delta,\ t_*\geq t_{*,0}\}$ be the domain bounded by $H_{I,0},H_F$ and $H_{I,1}$. Let us denote $X_{01}=\{r_0-\delta\leq r\leq r_1+\delta\}$ a slice of $\Omega_{01}^\circ$ of constant $t_*$. A better way of viewing $X_{01}$ is as the boundary at infinity of the compactification $\Omega_{01}$ of $\Omega_{01}^\circ$ at future infinity; see \eqref{EqPfSolComp} for a related construction. Thus, letting $\tau=e^{-t_*}$, $\Omega_{01}$ is the union of $\Omega_{01}^\circ$ and its boundary at future infinity $X_{01}\cong\{\tau=0\}\times X_{01}$. See Figure~\ref{FigArtificialExt}.

\begin{figure}[!ht]
  \centering
  \includegraphics{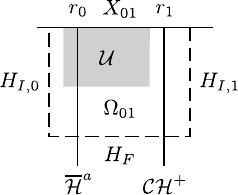}
  \caption{A neighborhood $\Omega_{01}$ of the artificial exterior region $r_0\leq r\leq r_1$, compactified at infinity, which is a submanifold with corners of the ambient spacetime $\wt M$. The shaded region $\cU$ is the elliptic set of the complex absorbing operator $\cQ$.}
\label{FigArtificialExt}
\end{figure}

One can then show using the methods of \cite{VasyMicroKerrdS,HintzVasySemilinear} that $\cW$ defines a Fredholm map
\[
  \cW \colon \cX_\bop^{s,\alpha} \to \cY_\bop^{s-1,\alpha}:=\Hb^{s-1,\alpha}(\Omega_{01})^{-,\bullet},
\]
where
\[
  \cX_\bop^{s,\alpha} = \{ u\in\Hb^{s,\alpha}(\Omega_{01})^{-,\bullet} \colon \cW u\in\cY_\bop^{s-1,\alpha} \},
\]
for all $\alpha\in\R$ outside of a discrete set of forbidden weights, where $s<1/2+\alpha/\kappa_1$, $\kappa_1$ the surface gravity of the Cauchy horizon at $r=r_1$. (The latter comes from estimates at saddle points for the null-geodesic flow in the b-cotangent bundle, see \cite[Proposition~2.1]{HintzVasySemilinear}.) Here, $\Hb^{s,\alpha}(\Omega_{01})^{-,\bullet}$ denotes the space of restrictions to $\Omega_{01}$ of elements of $\Hb^{s,\alpha}(\wt M)$ (see \eqref{EqPfSolComp} for the definition of $\wt M$) which are supported in $\{r_0-\delta\leq r\leq r_1+\delta\}$; there are no support restrictions at $H_F$. Thus, elements of $\Hb^{s,\alpha}(\Omega_{01}^\circ)^{-,\bullet}$ are supported distributions at $H_{I,0}\cup H_{I,1}$ and extendible distributions at $H_F$ in the sense of \cite[Appendix~B]{HormanderAnalysisPDE}.

The discrete set of forbidden weights consists of all $\alpha\in\R$ with are equal to $-\Im\sigma$ for a resonance $\sigma\in\C$ of $\cW$, i.e.\ a pole of the meromorphic continuation of
\[
  \wh\cW(\sigma)^{-1} \colon H^{s-1}(X_{01})^\bullet \to H^s(X_{01})^\bullet,\quad s<1/2-\Im\sigma/\kappa_1,
\]
where the $\bullet$ indicates distributions with supported character at $r=r_0-\delta$ and $r=r_1+\delta$. The key step in the proof of meromorphy is that
\begin{equation}
\label{EqPfSolAbsAnalFred}
  \wh\cW(\sigma) \colon \cX^s \to \cY^{s-1}:=H^{s-1}(X_{01})^\bullet,
\end{equation}
with $\cX^s=\{u\in H^s(X_{01})^\bullet\colon\wh\cW(\sigma)u\in\cY^{s-1}\}$, is an analytic family of Fredholm operators for $s<1/2-\Im\sigma/\kappa_1$; the regularity requirement comes from the radial point threshold at $r=r_1$, see \cite[Propositions~2.3 and 2.4]{VasyMicroKerrdS}. (The space $\cX^s$ here is independent of $\sigma$, since the principal symbol of $\wh\cW(\sigma)$ is.) Near a resonance $\sigma_0$, one can write
\begin{equation}
\label{EqPfSolAbsLaurent}
  \wh\cW(\sigma)^{-1} = \sum_{k=1}^{k_0}(\sigma-\sigma_0)^{-k}A_{-k} + A_0(\sigma),\quad A_{-k_0}\neq 0,
\end{equation}
where $A_{-k}=\sum_{\ell=1}^{\ell_k}\la\cdot,\psi_{k\ell}\ra\phi_{k\ell}$, and $A_0$ is holomorphic near $\sigma_0$; the integer $k_0$ is the order of the pole. All $A_j$, $j=-k_0,\ldots,0$, are bounded maps $\cY^{s-1}\to\cX^s$. We may assume that $\{\psi_{k\ell}\}_{\ell=1,\ldots,\ell_k}$ and $\{\phi_{k\ell}\}_{\ell=1,\ldots,\ell_k}$ are both linearly independent sets for each $k=1,\ldots,k_0$.

\begin{lemma}
\label{LemmaPfSolAbsResStates}
  Suppose $\sigma_0$ is a \emph{real} resonance. Then in the expansion \eqref{EqPfSolAbsLaurent}, we have $\ran A_{-k}\subset H^{1/2-0}(X_{01})^\bullet$, and moreover every element in $\ran A_{-k}$ has support in $r_{\cQ,-}\leq r\leq r_1$. Thus, all $\phi_{k\ell}$ are of this type.
  
  Moreover, all $\psi_{k\ell}\in\CI(X_{01})$.
\end{lemma}
\begin{proof}
  Let us write $\wh\cW(\sigma)=\sum_{j=0}^{k_0-1}(\sigma-\sigma_0)^j\cW_j + \cO(\sigma-\sigma_0)^{k_0}$; we note that $\cW_j\colon\cX^{1/2-0}\to\cY^{-1/2-0}$ since one can write
  \[
    \cW_j = \frac{1}{2\pi i}\oint_{|\sigma-\sigma_0|=\eps}\frac{\cW(\sigma)}{(\sigma-\sigma_0)^{j+1}}\,d\sigma,\quad \eps>0.
  \]
  Now, $\cW_0 A_{-k_0}=0$ implies $\ran A_{-k_0}\subset\ker\wh\cW(\sigma_0)$, which indeed is a subspace of $H^{1/2-0}(X_{01})^\bullet$, with its elements satisfying the support condition, since $\wh\cW(\sigma_0)$ is a hyperbolic (wave-type) differential operator in $r<r_{\cQ,-}$ and $r>r_1$ due to the spacelike nature of $\pa_{t_*}$ there. Suppose now the claim holds for $A_{-k_0},\ldots,A_{-k_0+\ell-1}$, $\ell\leq k_0-1$; then
  \[
    \cW_0 A_{-k_0+\ell} = - \sum_{j=1}^\ell \cW_j A_{-k_0+\ell-j},
  \]
  and the range of the operator on the right lies in $H^{-1/2-0}(X_{01})^\bullet$ (together with the support condition). Uniqueness for hyperbolic differential operators implies that the support of elements in $\ran A_{-k_0+\ell}$ is contained in $r_{\cQ,-}\leq r\leq r_1$, and the propagation of singularities, in particular radial point estimates at $r=r_1$, imply that $\ran A_{-k_0+\ell}\subset H^{1/2-0}(X_{01})$, as desired.

  The last claim follows from the proof of \eqref{EqPfSolAbsAnalFred} being Fredholm: Namely, the adjoint operator $\cW(\sigma)^*$ is Fredholm on \emph{high regularity} spaces, and the regularity estimates implying this yield that the dual resonant states $\psi_{k\ell}$ are all $\CI$ on $X_{01}$ (indeed, one can repeat the above argument for the first part, only now on high regularity spaces).
\end{proof}

For the following consequence of this lemma, we let $\sigma_1,\ldots,\sigma_N$ be the (finite!) collection of all real resonances, and we write
\begin{equation}
\label{EqPfSolAbsLaurent2}
\begin{gathered}
  \wh\cW(\sigma)^{-1}=\sum_{k=1}^{k_{j,0}}(\sigma-\sigma_j)^{-k}A_{j,-k}+A_{j,0}(\sigma)\tn{ near }\sigma=\sigma_j, \\
  A_{j,-k} = \sum_{\ell=1}^{\ell_{jk}} \la\cdot,\psi_{jk\ell}\ra\phi_{jk\ell},
\end{gathered}
\end{equation}
and we may assume that for fixed $j,k$, the $\psi_{jk\ell}$, $\ell=1,\ldots,\ell_{jk}$, are linearly independent, as are the $\phi_{jk\ell}$.

\begin{cor}
\label{CorPfSolAbsSupp}
  There exists $r'_{\cQ,+}\in(r_{\cQ,+},r_1)$ such that the following holds for the domain $\cV=\{r_{\cQ,-}<r<r'_{\cQ,+}\}\subset X_{01}$: For all $j,k,\ell$, we have $\supp\phi_{jk\ell}\cap\cV\neq\emptyset$, and moreover $\supp\psi\cap\cV\neq\emptyset$ for all $\psi\in\ker\wh\cW(\sigma)^*\subset\CI(X_{01})$.
\end{cor}
\begin{proof}
  Since there are only finitely many combinations of indices, it suffices to prove the statement for fixed $j,k,\ell$. Now, if the statement was false for $\phi_{jk\ell}$ for all $r'_{\cQ,+}$, then $\phi_{jk\ell}$, which we know vanishes for $r\leq r_{\cQ,-}$ and $r>r_1$, in fact vanishes everywhere except possibly at $r=r_1$; however, since $\phi_{jk\ell}\in H^{1/2-0}$, it cannot be a sum of differentiated delta distributions at $r=r_1$, so $\phi_{jk\ell}$ vanishes identically, which is absurd.

  Next, if $\psi$ with $\wh\cW(\sigma)^*\psi=0$ vanished in $r_{\cQ,-}<r<r_1$, then it would vanish in particular for $r\in[r_{\cQ,-},r_0)$, and it would vanish to infinite order at $r=r_1$; thus, by uniqueness for hyperbolic equations, it would have to vanish in $r<r_0$, and by unique continuation at $r=r_1$ (or weighted energy estimates) from the side $r>r_1$, where $\wh\cW(\sigma)^*$ is a wave operator, $\psi$ would also vanish in $r>r_1$. (See \cite[\S4]{VasyMicroKerrdS} or \cite{ZworskiRevisitVasy} for details.)
\end{proof}

Thus, replacing $r_{\cQ,+}$ by $r'_{\cQ,+}$, we may assume that the supports of all resonant states $\phi_{k\ell}$, as well as all dual states $\psi$ in the kernel of $\wh\cW(\sigma)^*$, have non-empty intersection with the region
\[
  \cV=\{r_{\cQ,-}<r<r_{\cQ,+}\}.
\]

We will now prove Lemma~\ref{LemmaPfSolAbsence} by perturbing $\cP$ by \emph{admissible} operators $\cR$:

\begin{definition}
  An operator $\cR\in\Psib^{-\infty}(\Omega_{01})$ is called \emph{admissible} if its Schwartz kernel is supported in $\ol\cU\times\ol\cU$, and if $\cR$ is dilation-invariant in $\tau$ (translation-invariant in $t_*$) near $\tau=0$, i.e.\ $\cR=N(\cR)$ near $\tau=0$.
\end{definition}

Our argument will rely on the stability of the Fredholm theory for $\wh\cW(\sigma)$ under perturbations, see \cite[\S2.7]{VasyMicroKerrdS} and \cite[Appendix~A]{HintzThesis}. We recall that the \emph{rank} of the pole $\sigma_j$ is defined by
\[
  \rank_{\sigma_j}\wh\cW(\sigma)^{-1} = \frac{1}{2\pi i}\tr\oint_{|\sigma-\sigma_j|=\delta} \wh\cW(\sigma)^{-1}\pa_\sigma\wh\cW(\sigma)\,d\sigma \in \N,
\]
where $\delta<\min_{k\neq j}|\sigma_k-\sigma_j|$. While the resonance at $\sigma_j$ may split up into several resonances upon perturbing $\cW$ to $\cW_\eps:=\cW+\eps\cR$, with $\cR$ admissible and $\eps\in\C$ small, the total rank of all resonances within $B_\delta(\sigma_j):=\{|\sigma-\sigma_j|<\delta\}$, given by the integral on the right with $\wh\cW$ replaced by $\wh{\cW_\eps}$, remains constant for small $\eps$. Thus, there are two cases: \emph{In the first case}, for some small $\eps$, the operator $\cW_\eps$ has at least $2$ resonances in $B_\delta(\sigma_j)$, in which case the ranks of all of these resonances must be strictly less than the rank of the resonance $\sigma_j$ of $\cW$; since the rank is an integer $\geq 1$, $\sigma_j$ can at most split up into finitely many resonances. We call $\sigma_j$ an \emph{unstable resonance}. \emph{In the second case}, for any admissible $\cR$ and any small $\eps$, $\cW_\eps$ has exactly one resonance $\sigma_j(\eps)\in B_\delta(\sigma_j)$, with $\sigma_j(0)=\sigma_j$, which necessarily has rank equal to $\rank_{\sigma_j}\wh\cW(\sigma)^{-1}$. In this case, we call $\sigma_j$ a \emph{stable resonance}.

By perturbing $\cW$ by an admissible $\cR$ (and calling the perturbation $\cW$ again, by an abuse of notation), we may therefore assume that all resonances of $\cW$ are stable. (We may have increased the number of real resonances in this process, but their number is still finite.) Under this assumption, we observe that for a (stable) resonance $\sigma_j$, and with $\sigma_j(\eps)$ the perturbed resonance of $\cW+\eps\cR$ as above, with $\cR$ admissible, we have
\[
  \sigma_j(\eps) = \frac{1}{2\pi i\rank_{\sigma_j}\wh\cW(\sigma)^{-1}}\tr\oint_{|\sigma-\sigma_j|=\delta} \sigma \wh{\cW_\eps}(\sigma)^{-1}\pa_\sigma\wh{\cW_\eps}(\sigma)\,d\sigma,
\]
which therefore depends analytically on $\eps$ in a neighborhood of $0$.

Next, we will arrange for the \emph{order} of a single resonance, say $\sigma_j$, to remain constant under small perturbations. More precisely, fix any admissible $\cR$, let $\cW_\eps:=\cW+\eps\cR$, and denote the perturbed resonance by $\sigma_j(\eps)$, where $\sigma_j(0)=\sigma_j$. For fixed $\delta>0$, let $E_\cR\subset\C$ be the largest connected open set containing $0$ for which no resonance of $\cW_\eps$ lies on $\pa B_\delta(\sigma_j)$ for any $\eps\in E_\cR$; in particular, $\sigma_j(\eps)\in B_\delta(\sigma_j)$ for all $\eps\in E_\cR$. We first claim that there exists a non-empty open subset $E'_\cR\subset E_\cR$ such that the order of $\sigma_j(\cR,\eps)$ is constant for $\eps\in E'_\cR$. Since the order is bounded above by the rank, which is constant and finite, this follows once we show that the order of $\sigma_j(\cR,\eps)$ is an upper semicontinuous function of $\eps$; now, if the pole order of $\wh{\cW_\eps}(\sigma)^{-1}$ at $\sigma=\sigma_j(\cR,\eps_0)$ equals $k$, then
\begin{equation}
\label{EqPfSolAbsPoleOrder}
  \frac{1}{2\pi i}\oint_{|\sigma-\sigma_j|=\delta} (\sigma-\sigma_j(\cR,\eps))^{k-1}\wh{\cW_\eps}(\sigma)^{-1}\,d\sigma \neq 0
\end{equation}
for $\eps=\eps_0$, and by continuity it will remain non-zero for nearby $\eps$; but \eqref{EqPfSolAbsPoleOrder} implies that the pole order at $\sigma_j(\cR,\eps)$ is at least $k$, proving the upper semicontinuity.

Now, let $k$ denote the maximal order of the pole $\sigma_j(\cR,\eps)$ over all admissible $\cR$ and all $\eps\in E_\cR$; again, $k$ is bounded by the rank of the resonance $\sigma_j$, hence the maximum is attained for some admissible $\cR_0$ and $\eps_0\in E_\cR$, and it is finite. Replacing $\cW$ by $\cW+\eps_0\cR_0$, we may now assume that for any fixed admissible $\cR$ and any small $\eps\in\C$, the order of the resonance $\sigma_j(\cR,\eps)$ of $\cW+\eps\cR$ is equal to $k$.

We now show how one can choose $\cR$ and $\eps$ so as to ensure $\sigma_j(\cR,\eps)\notin\R$: In the notation of \eqref{EqPfSolAbsLaurent2}, let $\phi_0:=\phi_{j,k_{j,0},1}\in\ker\wh\cW(\sigma_j)$, and pick any $\psi_0\in\ker\wh\cW(\sigma_j)^*$. By the paragraph following Corollary~\ref{CorPfSolAbsSupp}, $\supp\phi_0$ and $\supp\psi_0$ intersect $\cV$ non-trivially, and we can therefore pick an admissible operator $\cR$ such that
\begin{equation}
\label{EqPfSolAbsNonzero}
  \la\wh\cR(\sigma_j)\phi_0,\psi_0\ra\neq 0.
\end{equation}
This is the central point of the proof and critically relies on the support structure of resonant and dual states in Corollary~\ref{CorPfSolAbsSupp}.

Now for $f\in H^{-1/2-\delta}(X_{01})^\bullet$ to be chosen momentarily, define (dropping $\cR$ from the notation, now that it is fixed)
\[
  \phi(\eps) := \frac{1}{2\pi i}\oint_{|\sigma-\sigma_j|=\delta} (\sigma-\sigma_j(\eps))^{k-1}\wh{\cW_\eps}(\sigma)^{-1}f\,d\sigma,
\]
which lies in the range of the most singular Laurent coefficient of $\wh{\cW_\eps}(\sigma)^{-1}$ at $\sigma_j(\eps)$, thus $\phi(\eps)\in\ker\wh{\cW_\eps}(\sigma_j(\eps))$. We choose $f$ such that $\phi(0)=\phi_0$. Then, differentiating the equation $\wh{\cW_\eps}(\sigma_j(\eps))\phi(\eps)=0$ at $\eps=0$ and pairing with $\psi_0$ yields
\[
  0 = \sigma_j'(0)\la\pa_\sigma\wh\cW(\sigma_j)\phi_0,\psi_0\ra + \la\wh\cR(\sigma_j)\phi_0,\psi_0\ra + \la\wh\cW(\sigma_j)\phi'(0),\psi_0\ra.
\]
The last term vanishes by assumption on $\psi_0$, while the second term is non-zero; and since $\sigma_j(\eps)$ is analytic near $0$, $\sigma_j'(0)$ is finite. We conclude that $\sigma_j'(0)$ must be non-zero. Choosing $\eps$ very small and such that $\sigma_j'(0)\eps\notin\R$, we therefore have $\sigma_j(\eps)\notin\R$, which achieves our goal of perturbing the resonance $\sigma_j$ off the real axis.

Proceeding similarly with the remaining (finitely many) real resonances concludes the proof of Lemma~\ref{LemmaPfSolAbsence}.


\end{document}